\documentclass[11pt,twoside]{amsart}
\usepackage{graphicx} 
\usepackage{mathrsfs} 
\usepackage{dsfont}

\usepackage{fvextra}
\input{Macros}

\DeclareMathOperator{\fm}{\mathfrak{m}}
\DeclareMathOperator{\fn}{\mathfrak{n}}

\DeclareMathOperator{\mst}{M_\text{st}}
\DeclareMathOperator{\est}{e_\text{st}}

\DeclareMathOperator{\Est}{E_\text{st}}

\def\Kzero{K_0(\mathrm{Var}_k)}
\def\Kzerouh{K^{\mathrm{uh}}_0(\mathrm{Var}_k)}
\def\Muh{\widehat{\mathcal{M}}^{\mathrm{uh}}_k}

\title{On linear $\alpha_p$-quotients}
\author{Quentin Posva and Takehiko Yasuda,
with an appendix by Linus Rösler}
\date{}

\address{Institut de Mathématiques, Université de Neuchâtel, Rue Emile-Argand 11, 2000 Neuchâtel, Switzerland} 
\email{quentin.posva@unine.ch}

\address{Department of Mathematics, Graduate School of Science, the University
of Osaka, Toyonaka, Osaka 560-0043, Japan}
\address{Kavli Institute for the Physics and Mathematics of the Universe, The University of Tokyo, 5-1-5 Kashiwanoha, Kashiwa, Chiba, 277-8583, Japan}
\email{yasuda.takehiko.sci@osaka-u.ac.jp}

\address{\'Ecole Polytechnique F\'ed\'erale de Lausanne, Chair of Algebraic Geometry \newline 
\indent MA C3 615 (Bâtiment MA), Station 8, CH-1015 Lausanne}
\email{linus.rosler@epfl.ch}

\begin{document}

\maketitle

\begin{quote}
\textsc{Abstract.} We study ``linear" $\alpha_p$-actions on affine spaces and the associated quotient singularities, using explicit stacky resolutions. We describe when the quotient singularities are log canonical, canonical or terminal, and we compute their stringy motivic invariants. The second author and Fabio Tonini conjectured that these invariants coincide with those of linear $\bZ/p$-quotients: our approach reduces this conjecture to an equality of explicit multi-sets, which we check for a large number of primes using a computer software. 
A general proof of the equality of multi-sets is given in the appendix written by Linus Rösler.
\end{quote}

\tableofcontents

\section{Introduction}
Let $k$ be an algebraically closed field of positive characteristic $p>0$. In this article, we study ``linear" $\alpha_p$-quotient singularities from the points of view of the MMP and invariants steaming from motivic integration, and compare their properties to those of linear $\bZ/p$-quotients.

Quotients by $\alpha_p$-actions tend to exhibit a surprising behaviour that is sensitive to the ratio between their dimension and the characteristic $p$: their study is instructive by itself and in order to draw the line between characteristic $0$ theorems and positive characteristic phenomenons (see for example \cite{Posva_Pathological_MMP_sing}).
The study of the linear $\alpha_p$-quotient singularities was initiated in \cite{Tonini_Yasuda_Motivic_McKay_cor_for_alpha_p}, which was the initial motivation of this paper. The approach taken there is that of motivic integration, inspired by the study of linear $\bZ/p$-quotients singularities in \cite{Yasuda_p-cyclic_McKay_correspondence, Yasuda_Discrepancies_of_p_cyclic_quotient}. Our approach in this paper is different, but has applications to the questions raised in \cite{Tonini_Yasuda_Motivic_McKay_cor_for_alpha_p}.

To define these singularities, fix a sequence of non-negative integers $\mathbf{d}=\{d_\lambda\}_{\lambda=1,\dots,l}$ such that $0\leq d_\lambda \leq p-1$ and let $d=\sum_\lambda (1+d_\lambda)$. For a function $\varphi\in H^0(\bA^d,\sO_{\bA^d})$, let us write
        \begin{equation*}
                J_{\varphi,\mathbf{d}}=
                \begin{pmatrix}
                J_{\varphi,d_1} \\
                & J_{\varphi,d_2} \\
                && \ddots \\
                &&& J_{\varphi,d_l}
                \end{pmatrix}
                \end{equation*}
where $J_{\varphi,d_\lambda}$ is the $(d_\lambda+1)\times (d_\lambda+1)$ lower triangular Jordan block with eigenvalue (i.e.\ diagonal entries) $\varphi$. Consider the $k$-linear derivation $\partial_{\mathbf{d}}$ on $\bA^d_k$ defined through its action on coordinate functions $x_1,\dots,x_d$ as follows:
        $$\partial_\mathbf{d}\colon              
        \big(x_1,\dots,x_d\big)^\intercal 
        \mapsto 
        J_{0,\mathbf{d}}\cdot \big(x_1,\dots,x_d\big)^\intercal .
        $$
Concretely, $\partial_\mathbf{d}$ acts as a shift of indices in each block defined by the partition $\{d_\lambda\}$. This derivation defines an $\alpha_p$-action on $\bA^d$. Let us write
        $$\bA/(\alpha_p,\mathbf{d})=\bA^d_k/\langle \partial_{\mathbf{d}}\rangle.$$
Such quotients (and actions) will be called \emph{linear} $\alpha_p$-quotients (and actions).

\medskip
In this article, we study $\bA/(\alpha_p,\mathbf{d})$ using an explicit partial resolution. In \autoref{section:partial_resolution} we construct a weighted blow-up $b\colon \sY\to \bA^d$
where $\sY$ is a regular tame Deligne--Mumford stack such that the $1$-foliation $\sF_\sY$ on $\sY$ induced by the $\alpha_p$-action on $\bA^d$, is regular. This implies that $\sW=\sY/\sF_\sY$ is a regular tame DM stack that resolves $\bA/(\alpha_p,\mathbf{d})$. The coarse moduli space of $\sW$ has only toroidal singularities, and thus is a schematic partial resolution of the linear $\alpha_p$-quotient. A nice feature is that it is possible to compute the discrepancies of $\bA/(\alpha_p,\mathbf{d})$ from the discrepancies of $\sY$ and $\sF_\sY$. This is similar to the analysis done in \cite[\S 3]{Posva_Pathological_MMP_sing}, with some extra twists due to the presence of stacks: the necessary tools are developed in \autoref{section:DM_stacks} and \autoref{section:foliations}. In this way we obtain a description of the MMP singularities of linear $\alpha_p$-quotients:

\begin{theorem}[\autoref{thm:MMP_sing_alpha_p_qt}]\label{thm_intro:MMP_sing}
    Suppose that $\bD_\mathbf{d}=\sum_\lambda \frac{d_\lambda(d_\lambda+1)}{2}\geq 2$.
    Then $\bA/(\alpha_p,\mathbf{d})$ is lc if and only if $\bD_\mathbf{d}\geq p-1$, is canonical if and only if $\bD_\mathbf{d}\geq p$, and is terminal if and only if $\bD_\mathbf{d}\geq p+1$.
\end{theorem}

The singularity $\bA/(\alpha_p,\mathbf{d})$ is usually not Cohen--Macaulay (see \autoref{rmk:qt_not_CM}), so the above theorem gives new examples of canonical and terminal singularities in positive characteristic which are not Cohen--Macaulay, unlike what would happen in characteristic $0$. This adds to a growing list of examples
\cite{Yasuda_p-cyclic_McKay_correspondence, Gongyo-Nakamura-Tanaka, Totaro_Failure_Kodaira_vanishing, Yasuda_Discrepancies_of_p_cyclic_quotient, Totaro_Terminal_3folds_not_CM, Posva_Pathological_MMP_sing}. 

Our construction also gives a way to compute the motivic stringy invariant 
$\mst (\bA/(\alpha_p,\mathbf{d}))$---see \autoref{section:mst} for a short introduction to this invariant. 
In \autoref{section:singularities_partial_resolution} we give a stratification of the coarse moduli space of $\sW$ into equisingular loci, each of which has tame cyclic quotient singularities. A formula of Batyrev \cite{Batyrev_NA_integrals_and_stringy_Euler_numbers} allows us to compute the motivic stringy invariants of these strata. This leads to the following result:

\begin{theorem}[\autoref{thm:mst_alpha_p_qt}]
Assume that $\bD_\mathbf{d}>p-1$. Then
    $$\mst\big(\bA/(\alpha_p,\mathbf{d})\big)
=  \bL^d-\bL^l+  \frac{\bL^{l-1}}{1-\bL^{-1-\bD_\mathbf{d}+p}}\sum_{s/r\in F} (\bL^{N_r}-1) \bL^{\theta(s/r)}$$
where $F$ is the \emph{Farey sequence} of order $\max\mathbf{d}$, i.e.\ 
        $$F=\left( \bigcup_{j=1}^{\max\mathbf{d}}
        \frac{1}{j}\bZ \right)
        \cap (0;1],$$
and $\theta\colon (0;1]\to \bQ$ is the function defined by
        $$\theta(y)=1-\lfloor y\rfloor 
        +\lfloor py\rfloor
        -\sum_{(\sigma,i)}\lfloor iy\rfloor$$
where $(\sigma,i)$ runs through the elements of the set $\{(\sigma,i)\mid \sigma=1,\dots,l \text{ and } i=1,\dots,d_\sigma\}$.
\end{theorem}

There is a bijective correspondence between linear $\alpha_p$-actions and linear $\bZ/p$-actions. Indeed, any linear $\bZ/p$-action on $\bA^d$ is generated by an automorphism
        $$(x_1,\dots,x_d)^\intercal 
        \mapsto J_{1,\mathbf{d}}\cdot (x_1,\dots,x_d)^\intercal$$
for some $\mathbf{d}$ and some adequate choice of coordinates $x_1,\dots,x_d$. Let us write
                $$\bA/(\bZ/p,\mathbf{d})=\bA^d_k/\langle J_{1,\mathbf{d}}\rangle.$$
This description hints at possible similarities between 
$\bA/(\bZ/p,\mathbf{d})$ and $\bA/(\alpha_p,\mathbf{d})$---and this was the starting point of the article \cite{Tonini_Yasuda_Motivic_McKay_cor_for_alpha_p}. In fact, we explain in \autoref{section:degeneration} that there is an action of a Tate--Oort group on $\bA^d\times \bA^1_s$ over the base $\bA^1_s$, whose fiber over every $s\in k^\times$ is isomorphic to the linear $\bZ/p$-action given by $\mathbf{d}$, and whose fiber over $0$ is the linear $\alpha_p$-action given by $\mathbf{d}$. So linear $\alpha_p$-actions are degenerations of linear $\bZ/p$-actions, which gives another reason to expect that the respective quotient singularities share some properties and invariants.


There is one remarkable instance where this expectation is true. In \cite{Yasuda_p-cyclic_McKay_correspondence, Yasuda_Discrepancies_of_p_cyclic_quotient}, the second author of the present work described when $\bA/(\bZ/p,\mathbf{d})$ is lc, canonical or terminal: the conditions he found are exactly the ones we give in \autoref{thm_intro:MMP_sing} (\footnote{
    The reader will notice that our conventions regarding the bijection between arrays $\mathbf{d}$ and linear actions of $\bZ/p$ differ from those of \cite{Yasuda_p-cyclic_McKay_correspondence, Yasuda_Discrepancies_of_p_cyclic_quotient, Tonini_Yasuda_Motivic_McKay_cor_for_alpha_p}. We chose to depart from the conventions used in \emph{op.cit.} to facilitate computations with $\alpha_p$-actions. We explain how to navigate between the two possible conventions in \autoref{section:comparison}.
})!

Guided by the correspondence between linear $\alpha_p$ and $\bZ/p$-actions, the second author and Fabio Tonini conjectured in \cite[Conjecture 4.3]{Tonini_Yasuda_Motivic_McKay_cor_for_alpha_p} that $\bA/(\bZ/p,\mathbf{d})$ and $\bA/(\alpha_p,\mathbf{d})$ have the same stringy motivic invariant. We have closed formulas for both expressions (see the beginning of \autoref{section:comparison} for $\mst(\bA/(\bZ/p,\mathbf{d}))$, and using them we can reformulate the conjecture of Tonini--Yasuda as a conjecture of combinatorial nature, \autoref{conj:equality_of_multisets}. In the first version of this work, the first two authors obtained supporting evidence of this conjecture: equality of stringy Euler numbers (\autoref{prop:comparison_Euler_numbers}), and a proof-assisted verification in the conjecture for many cases up to $p\leq 173$ (\autoref{prop:comparison_mst}).
After the first version of this work appeared on the arXiv, Linus Rösler contacted the authors with an unconditional proof of \autoref{conj:equality_of_multisets}, which we added in \autoref{section:appendix}. So the conjecture of Tonini--Yasuda is fully verified:

\begin{theorem}[\autoref{section:comparison} and \autoref{section:appendix}]\label{conj:Tonini_Yasuda}
For all $\mathbf{d}$ we have:
        $$\mst\big(\bA/(\bZ/p,\mathbf{d})\big)
        = \mst\big(\bA/(\alpha_p,\mathbf{d})\big).$$
\end{theorem}

The verification of this conjecture raises other questions regarding the degeneration constructed in \autoref{section:degeneration}, see \autoref{question:simultaneous_resolution}.

\subsection{Acknowledgments}
This work starting during the conference-workshop ``P-adic and Characteristic p methods in Algebraic Geometry", held at the Bernoulli Center on EPFL campus, Switzerland, on 2-13 June 2025. We thank the organizers of that event for the opportunity of our meeting and the excellent working conditions.
The second author would like to thank Yudai Yamamoto for helpful conversations.
During this work, 
the second author 
was supported by JSPS KAKENHI Grant Numbers JP21H04994, JP23K25767 and JP24K00519.

\section{Preliminaries}

\subsection{Notations and conventions}
We work over an algebraically closed field $k$ of characteristic $p>0$. We follow the usual terminology of birational geometry, as in \cite{Kollar_Singularities_of_the_minimal_model_program}. We refer, for example, to \cite{Yasuda_Motivic_integration_over_DM_stacks} for the terminology used when working with DM stacks.


We use the following notation for tame cyclic quotients:
\begin{definition}\label{def:cyclic_qt}
Let $(\Spec(\sO),\sum_{i=1}^t \beta_iB_i)$ be a $\bQ$-factorial normal pair where $\sO$ is a local Noetherian $k$-algebra of dimension $t$, and where the $\beta_i$ are rational number. Let $n$ be a positive integer, let $a_1,\dots,a_t$ be finitely many, non-necessarily distinct elements of $\bZ$ . We say that $(\Spec(\sO),\sum_{i=1}^t \beta_iB_i)$ is of type
        $$\frac{1}{n}(a_1,\dots,a_t)$$
if there exists a $k$-linear $\mu_n$-action on $R=k\llbracket x_1,\dots, x_t\rrbracket$ such that:
    \begin{itemize}
        \item the maximal ideal of $R$ belongs to the fixed locus of the action;
        \item $x_i$ is an eigenfunction of $\mu_n$ of weight $[a_i]_n\in \bZ/n$;
        \item there is a $k$-linear isomorphism $R^{\mu_n}\cong \widehat{\sO}$;
        \item under the above isomorphism, the ideal of the divisor $B_i\otimes \Spec(\widehat{\sO})$ is equal to $(x_{i}R)\cap R^{\mu_n}$ for each $i$.
    \end{itemize}
\end{definition}

\subsection{Deligne--Mumford stacks}\label{section:DM_stacks}
We collect a few definitions and facts about DM stacks. Let $\sX$ be a DM stack that is reduced separated of finite type over $k$, with coarse moduli space $c\colon \sX\to X$.

\begin{definition}
A morphism $f\colon \sX\to \sY$ of DM stacks that is \emph{stabilizer-preserving} if: for every geometric point $\bar{x}\colon \Spec(\Omega)\to \sX$, the induced morphism of groups
        $$f\colon \Aut_\sX(\bar{x})\to \Aut_\sY(f\circ \bar{x})$$
is an isomorphism.
\end{definition}

Let us recall the following well-known \'{e}tale-local description of $\sX$ (see for example \cite[Theorem 2.12]{Olsson_Hom_stacks}). If $\bar{x}$ is a geometric point of $\sX$, then there exists a cartesian diagram
        \begin{equation}\label{eqn:local_presentation_DM_stacks}
            \begin{tikzcd}
            {[}\Spec(A)/G{]} \arrow[r, "\varpi"] \arrow[d] &
            \sX \arrow[d, "c"] \\
            \Spec(A^G)\arrow[r, "\pi"] & X
            \end{tikzcd}
        \end{equation}
where:
    \begin{itemize}
        \item $\pi$ and $\varpi$ are \'{e}tale and affine;
        \item $G=\Aut_\sX(\bar{x})$ acts on the local ring $(A,\fm)$;
        \item the induced $\Spec(A/\fm)\to \sX$ is a representative of $\bar{x}$, and $\varpi$ is stabilizer-preserving (in particular, $\Stab_G(\fm)=G$).
    \end{itemize}

For the notions of $\bQ$-Weil divisors of $\sX$ and (in case $\sX$ is normal) of the canonical divisor $K_\sX$, we refer to \cite[\S 4.1]{Yasuda_Motivic_integration_over_DM_stacks}. 
\begin{notation}
Let $\sD$ be a prime divisor on $\sX$, with generic point $\xi$. We write $r_\sX(\sD)=|\Aut_\sX(\xi)|$ the order of the residual gerbe at $\xi$.
\end{notation}

\begin{proposition}\label{prop:pullback_div_from_coarse_mod}
Suppose that $\sX$ is normal and \emph{tame}. Then:
    \begin{enumerate}
        \item $c^*K_X=K_\sX - \sum_{\mathcal{D}} (r_\sX(\mathcal{D})-1)\cdot\mathcal{D}$, where $\mathcal{D}$ runs through the prime divisors of $\sX$.
        \item If $E$ is a prime divisor of $X$, and $\mathcal{E}$ is the unique prime divisor of $\sX$ above $E$, we have $c^*E=r_\sX(\mathcal{E})\cdot \mathcal{E}$. 
    \end{enumerate}
\end{proposition}
\begin{proof}
This can be checked \'{e}tale-locally over $X$. So by \autoref{eqn:local_presentation_DM_stacks} we reduce to the case $\sX=[\Spec(A)/G]$ where $G$ is a finite abstract group of order invertible in $k$, in which case the two formulas are well-known.
\end{proof}

Finally, let us extend some standard vocabulary from birational geometry to DM stacks. Suppose that $\sX$ is normal and let $\mathcal{B}$ be a $\bQ$-Weil divisor on $\sX$ such that $K_\sX+\mathcal{B}$ is $\bQ$-Cartier. Let $f\colon \sY\to \sX$ be a birational proper \emph{representable} morphism from a normal DM stack. Then $f$ is an isomorphism over a big open subset of $\sX$ \cite[Proposition 4.2]{Kresch_Tschinkel_Birat_geom_DM_stacks} so the function $f^{-1}_*$ is well-defined on $\bQ$-Weil divisors of $\sX$. We can write
        \begin{equation}\label{eqn:discrep_for_DM_stacks}
        K_\sY+f^{-1}_*\mathcal{B}=f^*(K_\sX+\mathcal{B})
        +\sum_\mathcal{D}a(\mathcal{D};\sX, \mathcal{B})\cdot \mathcal{D},
        \quad a(\mathcal{D};\sX, \mathcal{B})\in \bQ,
        \end{equation}
where $\mathcal{D}$ runs through the $f$-exceptional divisors. The rational number $a(\mathcal{D};\sX,\mathcal{B})$ is called the \emph{discrepancy} of $\mathcal{D}$ with respect to $(\sX,\mathcal{B})$. The basic properties of discrepancies of schematic pairs also hold in this setting. If $U\to \sX$ is any \'{e}tale morphism from a scheme, then the formula \autoref{eqn:discrep_for_DM_stacks} pullbacks to the usual discrepancy formula for $(U,\mathcal{B}_U)$ along $f_U\colon \sY\times_\sX U \to U$: so if $D$ is an $f_U$-exceptional divisor, with image $\mathcal{D}$ in $\sY$, we have
        \begin{equation}\label{eqn:discrep_divisor_stacks}
            a(D;U,\mathcal{B}_U)=a(\mathcal{D};\sX,\mathcal{B}).
        \end{equation}

One way to construct such models of $\sX$ is as follows.
Let $g\colon Y\to X$ be a proper birational morphism from an irreducible normal variety $Y$. Let $\sY$ be the normalization of the main component of $\sX\times_XY$, with induced morphisms
    $$\begin{tikzcd}
        \sY\arrow[r, "c'"] \arrow[d, "f"] & Y \arrow[d, "g"] \\
        \sX\arrow[r, "c"] & X.
    \end{tikzcd}$$
Let $B$ be the $\bQ$-Weil divisor supported on $c(\mathcal{B})$ determined by the condition $c^*(K_X+B)=K_\sX+\mathcal{B}$.
Since $K_\sX+\mathcal{B}$ is $\bQ$-Cartier, so is $K_X+B$. In many cases we can compare the discrepancies of $(X,B)$ and $(\sX,\mathcal{B})$ along $g$ and $f$ respectively:

\begin{proposition}\label{prop:birat_model_of_stack}
In the above situation, $f$ is birational, proper and representable, and $c'$ is the coarse moduli morphism of $\sY$. Assume that $\sX$ is tame. Then $\sY$ is also tame and for every $f$-exceptional prime divisor $\mathcal{D}$ with image $D=c'(\mathcal{D})$, we have
        \begin{equation}\label{eqn:discrep_of_stacks_and_mod}
            a(\mathcal{D};\sX,\mathcal{B})+1= 
            r_\sY(\mathcal{D})\cdot
            [a(D;X,B)+1].
        \end{equation}
\end{proposition}
\begin{proof}
The morphism $c$ is a universal homeomorphism, and so is its base-change $\sX\times_XY\to Y$. In particular $\sX\times_YX$ is irreducible and $\sY$ is its normalization. The projection $\sX\times_XY\to \sX$ and the normalization morphism are  birational, proper and representable, so $f$ has the same properties. 

Let $\sY\to Y'$ be the coarse moduli of $\sY$: then $Y'$ is normal and we have a commutative diagram
        $$\begin{tikzcd}
        \sY\arrow[r] \arrow[d] & \sX\times_XY \arrow[d] \\
        Y' \arrow[r, "\gamma"] & Y
        \end{tikzcd}$$
where $\gamma$ is proper and quasi-finite, hence finite, and also birational. As $Y$ is normal, $\gamma$ is an isomorphism.

Assume that $\sX$ is tame: as $f$ is representable, $\sY$ is also tame. Assume furthermore that $K_\sX+\mathcal{B}$ is $\bQ$-Cartier. Then $K_X+B$ is $\bQ$-Cartier. We have $c^*(K_X+B)=K_\sX+\mathcal{B}$ and so 
        $$f^*(K_\sX+\mathcal{B})=c'^*g^*(K_X+B).$$
Around the generic point of an $f$-exceptional divisor $\mathcal{D}$ we have
        $$f^*(K_\sX+\mathcal{B})=K_\sY-a(\mathcal{D};\sX,\mathcal{B})\cdot \mathcal{D}$$
and
    \begin{eqnarray*}
        c'^*g^*(K_X+B) &=& c'^*\left(
        K_Y+g^{-1}_*B-a(D;X,B)\cdot D
        \right) \\
        &=& K_\sY-(r_\sY(\mathcal{D})-1)\cdot \mathcal{D}
        -a(D;X,B)\ r_\sY(\mathcal{D}) \cdot \mathcal{D}
    \end{eqnarray*}
where for the last equality we have used \autoref{prop:pullback_div_from_coarse_mod} and the fact that $\mathcal{D}$ is not contained in the support of $c'^*(g^{-1}_*B)$. Formula \autoref{eqn:discrep_of_stacks_and_mod} follows by comparing the two expressions.
\end{proof}

\subsection{1-foliations}\label{section:foliations}
We refer to \cite{Posva_Singularities_quotients_by_1-foliations, Posva_Resolution_1-foliations} and the references therein for the basic theory of $1$-foliations on schemes and Deligne--Mumford stacks. We will need some further observations. In the following, we let $\sX$ be a normal DM stack with affine diagonal that is seprated of finite type over $k$, we let $\sF$ be a $1$-foliation on $\sX$, and we let $q\colon \sX\to \sX/\sF$ denote the quotient morphism. Recall that $q$ is a representable, finite universal homeomorphism of height one.

\begin{proposition}\label{prop:qt_is_stab_preserving}
The quotient morphism $q\colon \sX\to \sX/\sF$ is stabilizer-preserving.
\end{proposition}
\begin{proof}
Let $\bar{x}$ be a geometric point of $\sX$. Let $G=\Aut_\sX(\bar{x})$, and $[\Spec(A)/G]\to \sX$ be a local presentation around $\bar{x}$ as in \autoref{eqn:local_presentation_DM_stacks}. If $\kappa$ is the residue field of $A$, then $G$ acts trivially on $\kappa$. Let $\sF_A$ be the pullback of $\sF$ to $\Spec(A)$: it is a $1$-foliation which is $G$-stable. Then the $G$-action descends to the local ring $(B=A^{\sF_A}, \mathfrak{n}=\fm\cap B)$. By construction of $\sX/\sF$, we have a cartesian diagram with \'{e}tale vertical morphisms
        $$\begin{tikzcd}
        \Spec(A) \arrow[r]\arrow[d] & \Spec(B) \arrow[d] \\
        \sX\arrow[r, "q"] & \sX/\sF
        \end{tikzcd}$$
and the induced groupoid $\Spec(B)\times_{\sX/\sF}\Spec(B)\rightrightarrows \Spec(B)$
is isomorphic over $\Spec(B)$ to the group action $G\times \Spec(B) \rightrightarrows \Spec(B)$ (see \cite[\S 3.2]{Posva_Resolution_1-foliations}, in particular Example 3.2.7 therein).
So if $\kappa'$ is the residue field of $B$, then $\Spec(\kappa')\to \sX/\sF$ is a representative of $f\circ\bar{x}$ and we have to show that the stabilizer of the closed point of $\Spec(B)$ inside $G$ is the whole group. This amounts to show that the induced action of $G$ on $\kappa'$ is trivial. The inclusion $B\hookrightarrow A$ is $G$-equivariant and $\fm\cap B=\fn$: so we get a $G$-equivariant inclusion $\kappa'\hookrightarrow \kappa$. As $G$ acts trivially on $\kappa$, it acts trivially on the sub-extension $\kappa'$.
\end{proof}

Another way to express the above result is the following:
\begin{proposition}\label{prop:qt_is_stab_preserving_II}
Let $G$ be an abstract group acting on a normal $k$-algebra $A$, and let $\mathscr{G}$ be a $G$-stable $1$-foliation on $U=\Spec(A)$. Then the quotient morphism $q\colon U\to U/\mathscr{G}$ is $G$-equivariant, and $\Stab_G(u)=\Stab_G(q(u))$ for every point $u\in U$.
\hfill \qedsymbol
\end{proposition}

\begin{remark}
Assume that $\sX$ and $\sF$ are regular and let $x\in \sX$ be a point with $\Stab_\sX(x)=\mu_n$. Then $X$ and the coarse moduli space of $\sX/\sF$ both have singularities of type $\frac{1}{n}(\dots)$ at the image of $x$, but the weights are usually different. For example, take $\sX=[\bA^2_{xy}/\mu_n]$ with weights $(m_x,m_y)$ and $\sF_{\bA^2}$ be the $1$-foliation on $\bA^2$ defined by $\partial_x$. Then $\sF$ descends to a regular $1$-foliation $\sF$ on $\sX$, and the coarse moduli space of $\sX/\sF$ has a singularity of type $\frac{1}{n}(m_x^p,m_y)$ at the image of the origin. More examples will be studied in \autoref{prop:sing_partial_resol}.
\end{remark}

In what follows we extend the discrepancy formulas of \cite[(4.2.6.f)]{Posva_Singularities_quotients_by_1-foliations} to the setting of stacks. This will be useful for the proof of \autoref{thm:MMP_sing_alpha_p_qt}.

We define the Weil divisor class $K_\sF$ on $\sX$ as follows: for every \'{e}tale morphism $f\colon U\to \sX$, with $U$ a scheme, we let $f^*K_\sF=K_{f^*\sF}$. 

Let $\mathcal{D}$ be a prime divisor on $\sX$. We say that $\mathcal{D}$ is $\sF$-invariant if: for every \'{e}tale morphism $f\colon U\to \sX$ with $U$ a scheme, every (equivalently, some) prime divisor $D$ of $U$ mapping to $\mathcal{D}$ is $f^*\sF$-invariant. We let
        $$\epsilon_\sF(\mathcal{D})=
        \begin{cases}
            0 & \text{if }\mathcal{D}\text{ is }\sF\text{-invariant,}\\
            1 & \text{otherwise.}
        \end{cases}$$
Given a $\bQ$-Weil divisor $\Delta=\sum_i a_i\Delta_i$ on $\sX$, we let
        \begin{equation}\label{eqn:natural_divisor_on_qt}
        \Delta_{\sX/\sF}=\sum_i a_i p^{-\epsilon_\sF(\Delta_i)} q(\Delta_i).
        \end{equation}
\begin{proposition}[Log adjunction formula]\label{prop:log_adj_formula}
With the notations as above, we have
    $$q^*(K_{\sX/\sF}+\Delta_{\sX/\sF})=K_\sX+\Delta+(p-1)K_\sF.$$
\end{proposition}
\begin{proof}
This can be checked \'{e}tale-locally on $\sX$, and so the formula follows from the usual schematic one \cite[Proposition 4.2.3]{Posva_Singularities_quotients_by_1-foliations}.
\end{proof}

Now let $f\colon \sY\to \sX$ be a birational and \emph{representable} morphism where $\sY$ is a normal DM stack. Pick an \'{e}tale presentation $U\to \sX$ where $U$ is a scheme, and let $V=U\times_\sX \sY$. We have a commutative diagram
        \begin{equation}\label{eqn:pullback_foliation_stack}
        \begin{tikzcd}
        V\times_\sY V \arrow[r, shift left=1, "s" above]\arrow[r, shift right=1, "t" below] \arrow[d, "f_1"]
        & V \arrow[r, "\pi"] \arrow[d, "f_0"] 
        & \sY \arrow[d, "f"] \\
        U\times_\sX U \arrow[r, shift left=1]\arrow[r, shift right=1] & U \arrow[r] & \sX,
        \end{tikzcd}
        \end{equation}
where $f_1$ is induced by the two morphisms $f_0\circ s,f_0\circ t\colon V\times_\sY V\to U$.
The morphism $f_0$ is representable and $\sY$ is normal: in particular $V$ and $V\times_\sY V$ are normal schemes. The projection $\pi\colon V\to \sY$ is \'{e}tale and surjective, so it is a presentation. Notice that $f_1$ is birational. Let $\sF_U$ and $\sF_{U\times_\sX U}$ be the $1$-foliation pullbacks of $\sF$ to $U$ and $U\times_\sX U$ respectively. Since $f_0$ and $f_1$ are birational, we can define the $1$-foliations $\sF_V=f_0^*\sF_U$ and $\sF_{V\times_\sY V}=f_1^*\sF_{U\times_\sX U}$. We have
        $$s^*\sF_V=\sF_{V\times_\sY V}=t^*\sF_V$$
generically on $V\times_\sY V$, hence globally as $1$-foliations (cf.\ \cite[Remark 2.4.2]{Posva_Singularities_quotients_by_1-foliations}). So $\sF_V$ descends to a $1$-foliation $f^*\sF$ on $\sY$.

Assume furthermore that $K_\sF$ is $\bQ$-Cartier. Then we can write
        $$K_{f^*\sF}=f^*K_\sF
        +\sum_{\mathcal{D}} a(\mathcal{D};\sF)\cdot \mathcal{D}, 
        \quad 
        a(\mathcal{D};\sF)\in \bQ,$$
where $\mathcal{D}$ runs through the $f$-exceptional divisors. Pulling back this formula along $\pi$, we obtain
        $$K_{f_0^*\sF_U}=f_0^*K_{\sF_U}+
        \sum_{\mathcal{D}} a(\mathcal{D};\sF)\cdot \pi^*\mathcal{D}.$$
As $\pi^*\mathcal{D}$ is the sum of the prime divisors of $V$ mapping to $\mathcal{D}$ (with coefficients one), we find
        \begin{equation}\label{eqn:discrep_foliation_on_stack}
        a(D;\sF_U)=a(\mathcal{D};\sF)
        \end{equation}
for every $f_0$-exceptional divisor $D$ with image $\mathcal{D}$ in $\sY$.

\begin{proposition}\label{prop:reg_fol_is_can}
Assume that $\sX$ and $\sF$ are regular. Then $a(\mathcal{D};\sF)\geq 0$ for every $\mathcal{D}$.
\end{proposition}
\begin{proof}
By assumption $U$ and $\sF_U$ are regular (cf.\ \cite[Definition 3.1.5]{Posva_Resolution_1-foliations}). Writing
        $$K_{f_0^*\sF_U}=f_0^*K_{\sF_U}+\sum_D a(D;\sF_U)\cdot D$$
where $D$ runs through the $f_0$-exceptional divisors, we have $a(D;\sF_U)\geq 0$ for every $D$ by \cite[Lemma 3.0.3]{Posva_Singularities_quotients_by_1-foliations}. The divisor $\pi^*\mathcal{D}$ is supported on $f_0$-exceptional divisors, so the statement follows from \autoref{eqn:discrep_foliation_on_stack}.
\end{proof} 

Finally, we have a commutative diagram
                \begin{equation}\label{eqn:comm_square_of_qt}
                \begin{tikzcd}
                    \sY\arrow[r, "q'"] \arrow[d, "f"] 
                    & \sY/f^*\sF \arrow[d, "g"] \\
                    \sX \arrow[r, "q"] & \sX/\sF
                \end{tikzcd}
                \end{equation}
where the horizontal arrows are the quotient morphisms, and where $g$ is birational and representable. This follows from the diagram \autoref{eqn:pullback_foliation_stack} and the construction of infinitesimal quotients for DM stacks \cite[\S 3.2]{Posva_Resolution_1-foliations}.
 
\begin{theorem}\label{thm:discrep_infinit_qt_stack}
Let $\Delta$ be a $\bQ$-Weil divisor on $\sX$ such that $K_\sX+\Delta$ is $\bQ$-Cartier. Let $\mathcal{D}$ be an $f$-exceptional divisor, with image $\mathcal{E}$ in $\sY/f^*\sF$. Then:
        $$a(\mathcal{E}; \sX/\sF, \Delta_{\sX/\sF})
        =
        p^{-\epsilon_{\sY}(\mathcal{D})}\cdot \big[  a(\mathcal{D};\sX,\Delta)
        +(p-1)\cdot a(\mathcal{D};\sF)
        \big].$$
\end{theorem}
\begin{proof}
The same argument as in the proof of \cite[Theorem 4.2.5]{Posva_Singularities_quotients_by_1-foliations} applies, using the commutative square \autoref{eqn:comm_square_of_qt} and \autoref{prop:log_adj_formula}. Alternatively, we deduce the statement from \cite[(4.2.6.f)]{Posva_Singularities_quotients_by_1-foliations} applied to the commutative square
        $$\begin{tikzcd}
        V\arrow[d, "f_0"] \arrow[r] & V/\sF_V \arrow[d] \\
        U \arrow[r] & U/\sF_U
        \end{tikzcd}$$
and from the formulas \autoref{eqn:discrep_divisor_stacks} and \autoref{eqn:discrep_foliation_on_stack}.
\end{proof}

\subsection{Stringy motivic invariant}\label{section:mst}
We provide a short introduction to the stringy motivic invariants, and refer to \cite{Yasuda_p-cyclic_McKay_correspondence} for more details.
The Grothendieck ring of varieties, denoted by $\Kzero$, is defined as the quotient of the free abelian group generated by isomorphism classes of varieties modulo the scissor relation. We consider the following two variants of this ring. Firstly, we define $\Kzerouh$ to be the quotient of $\Kzero$ by the extra relation that for each universal homeomorphism $Y\to X$, $[Y]=[X]$ (see \cite[p.115]{Chambert-Loir_Nicaise_Sebag_Motivic_integration}). Let $\bL =[\bA^1]$ be the class of an affine line. We then define a version of the complete Grothendieck ring of varieties by
localizing $\Kzerouh$ with $\bL$ and then taking the dimensional completion; we denote this complete ring by $\Muh$. We need this variant of the complete Grothedieck ring not for establishing a general theory of motivic integration in positive characteristic, but for giving a combinatorial formula for stringy motivic invariants of quotients $\bA^d/\alpha_p$.

Let $(X,B)$ be a klt pair such that $K_X+B$ is Cartier. The divisor $B$ can have only non-positive coefficients. Since $\omega_X(B)$ is invertible, we can define an ideal sheaf $I$ by
$$
Im \left(\bigwedge^d \Omega_X\to \omega_X(B)\right) = I\cdot \omega_X (B).
$$
The stringy motivic invariant $\mst(X,B)$ is defined as 
$$
\mst(X,B) = \int_{J_\infty X} \bL^{\ord I} \,d\mu _X \in \Muh. 
$$
More generally, for a locally closed subset $C$, we define the stringy motivic invariant along $C$ by
$$
\mst(X,B)_C = \int_{(J_\infty X)_C} \bL^{\ord I} \,d\mu _X \in \Muh. 
$$
Here $(J_\infty X)_C$ denotes the space of arcs that intersect with $C$. 

In fact, the invariant $\mst(X,B)_C$ depends only on singularities of the pair $(X,B)$ along $C$. More precisely, if $C$ is a closed subset in an open subset $U\subset X$ then $\mst(X,B)_C$ depends only on the isomorphism class of the pair $(\widehat{U}, B|_{\widehat{U}})$ obtained by taking the formal completion of $U$ along $C$ and restricting $B$ to $\widehat{U}$. This is because the space $(J_\infty X)_C$ together with the motivic measure and the function $\ord I$ depends only on the isomorphism class of $(\widehat{U}, B|_{\widehat{U}})$. 
In particular, when $C=\{x\}$ with $x$ a closed point, then $\mst(X,B)_x$ is determined by $(\Spec \widehat{\mathcal{O}}_{X,x},B|_{\Spec \widehat{\mathcal{O}}_{X,x}})$. Thus, we sometimes write this as $\mst(\Spec \widehat{\mathcal{O}}_{X,x},B|_{\Spec \widehat{\mathcal{O}}_{X,x}})_x$. 
The following are some more basic properties of stringy motivic invariants:

\begin{proposition}
Let $(X,B)$ be a klt pair as above. 
\begin{enumerate}
    \item If a locally closed subset $C\subset X$ is the disjoint union of finitely many locally closed subsets $C_1,\dots,C_l$, then $\mst(X,B)_C=\sum_{i=1}^l \mst(X,B)_{C_i}$.
    \item If $f\colon (Y,B')\to (X,B)$ is a proper birational crepant morphism of klt pairs (so that $K_Y+B'$ is also Cartier) and if $C\subset X$ is a locally closed subset, then $\mst(Y,B')_{f^{-1}(C)}=\mst(X,B)_C$. 
\end{enumerate}
\end{proposition}

\begin{remark}
When $K_X+B$ is only $\bQ$-Cartier, then the stringy motivic invariant is defined as an element of a version of the complete Grothendieck ring with some fractional power of $\bL$ adjoined. 
\end{remark}

\section{Degeneration of linear $\bZ/p$-actions into $\alpha_p$-actions}\label{section:degeneration}
We keep the notations of the Introduction. In this section we explain how the $\bZ/p$-action on $\bA^d$ given by $J_{1,\mathbf{d}}$ can be degenerated over the affine line $\bA^1_s$ into the $\alpha_p$-action given by $J_{0,\mathbf{d}}$. Informally, if we replace the upper diagonal of $J_{1,\mathbf{d}}$ by the parameter $s$, then for $s\neq 0$ we obtain a family of $\bZ/p$-actions whose first-order deviation from the identity is $s\cdot J_{0,\mathbf{d}}$.

To make this precise, we start by introducing a version of the \emph{Tate--Oort group scheme} \cite{Reid_Tate-Oort_group}. The base of the group scheme will be $B=\{st=0\}\subset \bA^2_{st}$. The group is given by
        $$\sG=\Spec_B \left(
        \frac{\sO_B[x, (1+tx)^{-1}]}{(x^p-s^{p-1}x)}\right)$$
with group law over $B$ given by the co-multiplication
        $$\sO(\sG)\to \sO(\sG)\otimes \sO(\sG), \quad 
        x\mapsto x\otimes 1+ 1\otimes x +t x\otimes x.$$

\begin{proposition}[cf.\ {\cite[Proposition 3.1]{Reid_Tate-Oort_group}}]
With the notations as above,
    \begin{enumerate}
        \item $\sG$ is a finite flat group scheme of length $p$ over $B$,
        \item $\sG\otimes_B \{s\neq 0\}$ is isomorphic to the constant group $\bZ/p$;
        \item $\sG\otimes_B \{t\neq 0\}\cong \mu_{p, \{t\neq 0\}}$;
        \item $\sG\otimes_B\{s=t=0\}\cong \alpha_p$.
    \end{enumerate}
\end{proposition}

Let $\sG_a=\sG\otimes_B\{t=0\}$: it is a finite flat group scheme over $\bA^1_s$ that is the constant group $\bZ/p$ over the open set $\{s\neq 0\}$, and whose fiber above $s=0$ is $\alpha_p$. (Here $a$ stands for ``additive", as $\sG_a$ embeds inside the additive group over $\bA^1_s$. In \cite[Chapter 7]{Miyanishi_Ito_Algebraic_surfaces_in_pos_char}, $\sG_a$ is called the \emph{unified $p$-group scheme}.)

Consider the following matrices: we let
        $$\widetilde{J}_{s,d_\lambda}=
        \begin{pmatrix}
        1 &  \\
        s & 1 & \\
        & \ddots & \ddots \\
        && s & 1 &  \\
        &&& s & 1
        \end{pmatrix}
        \in M_{(d_\lambda+1)\times (d_\lambda+1)}(k[s])$$
and
        \begin{equation*}
        \widetilde{J}_{s,\mathbf{d}}=
        \begin{pmatrix}
            \widetilde{J}_{s,d_1} \\
            & \ddots \\
            && \widetilde{J}_{s,d_l}
        \end{pmatrix}.
        \end{equation*}
Observe that for $n\in \bZ$,
        $$(\widetilde{J}_{s,d_\lambda})^n=
        \begin{pmatrix}
        1 \\
        \binom{n}{1}s & 1 \\
        \binom{n}{2}s^2 & \binom{n}{1}s & 1 \\
        \vdots & \vdots & \ddots & \ddots  \\
        &&& \binom{n}{1}s & 1
        \end{pmatrix}
        \in M_{(d_\lambda+1)\times (d_\lambda+1)}(k[s])$$
and that 
        \begin{equation*}
        \widetilde{J}_{s,\mathbf{d}}^n=
        \begin{pmatrix}
            \widetilde{J}_{s,d_1}^n \\
            & \ddots \\
            && \widetilde{J}_{s,d_l}^n
        \end{pmatrix},
        \end{equation*}
and that all these matrices have order $p$.
If $f\colon T\to \bA^1_s$ is any $\bA^1_s$-scheme, we have the matrices
$f^*(\widetilde{J}_{s,\mathbf{d}}^n)\in \GL_{T}(\bA^d\times T)$
and it is easy to check that 
$f^*(\widetilde{J}_{s,\mathbf{d}}^n)=(f^*\widetilde{J}_{s,\mathbf{d}})^n$.
Therefore we can define a morphism of group functors
        $$\mathfrak{J}\colon \underline{\bZ/p}_{\bA^1_s}\longrightarrow \underline{\GL}_{\bA^1_s}(\bA^d\times \bA^1_s)$$
as follows: if $T$ is connected we let
        $$\mathfrak{J}_T(n)=f^*\widetilde{J}_{s,\mathbf{d}}^n \quad
        \text{for }n\in \underline{\bZ/p}_{\bA^1_s}(T)=\bZ/p,$$
and the value on disconnected schemes is given by the product over the values of connected components. So $\mathfrak{J}$ is a morphism of group schemes, whose image $\mathbf{G}$ is a closed sub-group scheme of $\GL_{\bA^1_s}(\bA^d\times\bA^1_s)$. We think of $\mathbf{G}$ as the sub-group generated by $\widetilde{J}_{s,\mathbf{d}}$ inside the matrix group $\GL(\bA^d\times\bA^1_s)$. It is clear that $\mathbf{G}$ is reduced, and finite flat of length $p$ over $\bA^1_s$.

\begin{proposition}\label{prop:degen_of_Z/p_to_alpha_p}
We have an isomorphism of $\bA^1_s$-group schemes $\mathbf{G}\cong \sG_a$. The fiber of the action of $\mathbf{G}$ on $\bA^d\times\bA^1_s$ over $s=1$ is the $\bZ/p$-action given by the matrix $J_{1,\mathbf{d}}$, and the fiber over $s=0$ is the $\alpha_p$-action given by the derivation $\partial_{\mathbf{d}}$.
\end{proposition}
\begin{proof}
Let $z\in \Gamma(\mathbf{G},\sO_\mathbf{G})$ be the function that corresponds to the entry on the lower diagonal of the Jordan blocks: generically over $\bA^1_s$ it assumes the values $0,s,\dots,(p-1)s$ and thus verifies the equation $z^p-s^{p-1}z=0$ globally. Since any $M\in \mathbf{G}(\Spec(A))$ is uniquely determined by the value $z(M)\in A$, and since $\mathbf{G}$ has length $p$ over $\bA^1_s$, we deduce that
        $$\mathbf{G}\cong \Spec_{\bA^1_s} \frac{k[s,z]}{(z^p-s^{p-1}z)} \quad \text{as }\bA^1_s\text{-schemes}.$$
If $M,M'\in \mathbf{G}(\Spec(A))$ are two $A$-valued points, then $z(M\cdot M')=z(M)+z(M')$. So the co-multiplication of $\Gamma(\mathbf{G},\sO_\mathbf{G})$ under the above isomorphism is given by $z\mapsto z\otimes 1 + 1\otimes z$. Hence $\mathbf{G}\cong \sG_a$ as group schemes over $\bA^1_s$.

The description of the fiber of the action of $\sG_a$ on $\bA^d\times\bA^1_s$ above $s=1$ is immediate. For $s=0$, we know that we have an $\alpha_p$-action. To find out the corresponding derivation $\partial$, we look at $k[\varepsilon]$-points of $\mathbf{G}$ above $s=0$, where $k[\varepsilon]$ is the ring of dual numbers: by letting $z=\varepsilon$ we find the infinitesimal automorphism
        $$(x_1,\dots,x_d)^\intercal 
        \mapsto \widetilde{J}_{\varepsilon , \mathbf{d}}\cdot (x_1,\dots,x_n)^\intercal
        =
        ((\id+\varepsilon\partial)(x_1),\dots,(\id+\varepsilon\partial)(x_n))^\intercal.$$
This shows that $\partial=\partial_{\mathbf{d}}$, as desired.
\end{proof}

Consider the geometric quotient $\sX=(\bA^d\times\bA^1_s)/\mathbf{G}$: it is an affine normal flat $\bA^1_s$-scheme of finite type. For each $\alpha\in \bA^1_s(k)$ we have a finite comparison morphism
        $$\varphi_\alpha\colon \bA^d/\mathbf{G}_\alpha 
        \longrightarrow \sX_\alpha$$
which is an isomorphism on the locus where $\mathbf{G}_\alpha$ acts freely. If $\bD_\mathbf{d}\geq 2$ then the ramification locus of the $\mathbf{G}$-action has relative codimension $\geq 2$ over $\bA^1_s$, and it follows that $\varphi_\alpha$ is an isomorphism if and only if $\sX_\alpha$ is $S_2$. This is the case for a general value of $\alpha$: for the set of points $x$ of $\sX$ such that $x$ is an $S_2$ point of its fiber over $\bA^1_s$, is open \cite[12.1.6]{EGA_IV.3} and contains the generic fiber. The image in $\bA^1_s$ of its complement is constructible and does not contain the generic point: so it must be a finite union of closed points. Furthermore, we have the $\bG_m$-action by scaling on $\bA^1_s$: it lifts to an action on $\bA^d\times\bA^1_s$ (acting trivially on the first factor) that commutes with the $\mathbf{G}$-action. So it descends to a $\bG_m$-action on $\sX$ that is transitive on the set of fibers $\sX_\alpha$ with $\alpha\neq 0$. So all these fibers are $S_2$, and it follows that $\varphi_\alpha$ is an isomorphism when $\alpha\neq 0$. Hence the family $\sX\otimes_{\bA^1_s}\{s\neq 0\} \to \{s\neq 0\}$ is fiberwise trivial. Indeed, we have
    $$\sX_\alpha\cong \bA^d/\mathbf{G}_\alpha\cong 
    \bA/(\bZ/p,\mathbf{d})$$ 
when $\alpha\neq 0$: for linear $\bZ/p$-actions on $\bA^d$ are uniquely determined, up to linear isomorphisms, by the dimensions and the multiplicities of the irreducible sub-spaces. 

The morphism $\varphi_0$ is an isomorphism if and only if the $\mathbf{G}$-cohomology group $H^1(\mathbf{G};\sO_{\bA^d\times \bA^1_s})$, seen as a $k[s]$-module, has only trivial $s$-torsion (see \cite[Exposé I, \S 5]{SGA3_I} for group-scheme cohomology). If one knows explicit generators of $H^0(\bA^d/\mathbf{G}_0,\sO)$, this can also be addressed by explicit computations: see \cite[\S 9.2]{Liedtke_Martin_Matsumoto_Lrq_sing} for concrete examples in a similar set-up. We will not study this question in more details here; nonetheless, the results that we will obtain in \autoref{section:comparison} suggest that $\varphi_0$ is an isomorphism.

In any case, as $\mathbf{G}$ is geometrically reductive over $\bA^1_s$, it follows that $\varphi_0$ is a finite universal homeomorphism. So $\bA/(\alpha,\mathbf{d})$ is a one-parameter deformation of $\bA/(\bZ/p,\mathbf{d})$ up to universal homeomorphism.

\begin{remark} 
One can ask whether the linear $\alpha_p$-action can be obtained as a degeneration in the other direction, i.e.\ as the limit of $\mu_p$-actions inside an action of 
    $$\sH=\sG\otimes_B \{s= 0\}=\Spec_{\bA^1_t}\frac{k[t,x,(1+tx)^{-1}]}{(x^p)}.$$ 
We indicate one way to do so.

The global section $x\in \Gamma(\sH,\sO_\sH)$ induces a group morphism $\sH\to \bG_a$, and the global section $1+tx$ induces a group morphism $\sH\to \bG_m$. These two morphisms yield a closed embedding of $\bA^1_t$-group schemes
        $$\sH\cong \{(y, 1+ty)\mid y^p=0\}\hookrightarrow \bG_a\times_{\bA^1_t} \bG_m.$$
Using this description, we can write down the following action of $\sH$ on $\bA^d\times\bA^1_t$. Let $\mathfrak{S}=\{d_1+1,d_1+d_2+2,\dots,d\}$ and define:
        $$(y,1+ty)\colon\quad  x_i\mapsto \begin{cases}
            x_i & \text{if }i\in \mathfrak{S}, \\
            (1+ty)x_i+yx_{i+1} & \text{if }i\notin \mathfrak{S}.
        \end{cases}$$
When $t=0$ and $y=\varepsilon$ this is the infinitesimal automorphism $\id+\varepsilon \partial_\mathbf{d}$, as desired.

Let $\sY=(\bA^d\times\bA^1_t)/\sH$. As before, we have finite comparison morphisms
        $$\psi_\beta\colon \bA^d/\sH_\beta\longrightarrow \sY_\beta, \quad 
        \beta\in \bA^1_t.$$
Since $\sH\otimes \{t\neq 0\}\cong \mu_p$ is linearly reductive, $\psi_\beta$ is an isomorphism for $\beta\neq 0$. The morphism $\psi_0$ is a finite universal homeomorphism, but in general we do not know whether it is an isomorphism.

Unlike the singularities of $\bA/(\bZ/p,\mathbf{d})$ and $\bA/(\alpha_p,\mathbf{d})$, those of $\bA^d/\sH_\beta$ for $\beta\neq 0$ are well-behaved for any $\mathbf{d}$: there are always toroidal, so in particular Cohen--Macaulay, rational and klt.
\end{remark}

\section{Partial resolutions of linear $\alpha_p$-quotients}\label{section:partial_resolution}
In this section we describe partial resolutions of linear $\alpha_p$-quotients, and derive some consequences. 

\subsection{Notations}\label{section:notations}
We start by fixing our notations. Let $\mathbf{d}=\{d_\sigma\}_{\sigma=1,\dots,l}$ with $0\leq d_\sigma \leq p-1$ and $d=\sum_{\sigma=1}^l(1+d_\sigma)$. 
We write $\bD_\mathbf{d}=\sum_{\sigma=1}^l \frac{d_\sigma(d_\sigma+1)}{2}$.

We consider the affine space $\bA^d$ with coordinates $\{\mathbf{x}^{(\sigma)}\}_\sigma$ where $\sigma=1,\dots,l$ and $\mathbf{x}^{(\sigma)}$ denotes the set of variables $x_0^{(\sigma)},\dots,x_{d_\sigma}^{(\sigma)}$. We define the derivation $\partial=\partial_{\mathbf{d}}$ by
        \begin{equation}\label{eqn:derivation}
            \partial =\sum_{\sigma=1}^l\partial^{(\sigma)} 
            =\sum_\sigma \left(\sum_{i=1}^{d_\sigma} x_{i}^{(\sigma)}\partial_{x_{i-1}^{(\sigma)}}\right)
        \end{equation}
where $\partial_{x_i^{(\sigma)}}$ denotes the derivative with respect to the variable $x_i$. (If $d_\sigma=0$, it is understood that $\partial^{(\sigma)}=0$.)

We let $\mathbf{I}=\{(\sigma,i)\mid \sigma=1,\dots,l \text{ and }i=1,\dots,d_\sigma\}$ and 
$\mathbf{I}^*=\{(\sigma,i)\in \mathbf{I}\mid i\neq 0\}$.

To fix the ideas: if $l=1$ and $d=3$, then $\partial=x_2\partial_{x_1}+x_1\partial_{x_0}$, corresponding to the matrix
        $$\begin{pmatrix}
        0 \\
        1& 0 \\
        & 1 & 0
        \end{pmatrix}$$
in the ordered basis $(x_0,x_1,x_2)$. 

It is easily seen that the $\partial^{(\sigma)}$s commute with each other, and that $\partial^p=\sum_\sigma(\partial^{(\sigma)})^p=0$. So $\partial$ defines an action of $\alpha_p$ on $\bA^d$, and $\sO_{\bA^d}\cdot \partial=\sF$ is a $1$-foliation on $\bA^d$. 

\begin{lemma}\label{lemma:sing_locus_of_qt}
The fixed-point scheme of the $\alpha_p$-action given by \autoref{eqn:derivation} on $\bA^d$ (or, equivalently, the locus where $\sF$ is not a regular $1$-foliation) is the closed sub-scheme defined by the ideal $\left(x_i^{(\sigma)}\mid (\sigma,i)\in\mathbf{I}^* \right)$.
\end{lemma}
\begin{proof}
This follows from \cite[Lemma 2.3.10]{Posva_Singularities_quotients_by_1-foliations}.
\end{proof}

\begin{lemma}
Assume that $\bD_\mathbf{d}\leq 1$. Then $\bA/(\alpha_p,\mathbf{d})$ is regular.
\end{lemma}
\begin{proof}
By definition of $\bD_\mathbf{d}$ it follows that $d_\sigma\leq 1$ with equality for at most one index. The action is trivial if $d_\sigma=0$ for all $\sigma$, in which case the quotient morphism is the identity. If $d_{\sigma_1}=1$, then the kernel of $\partial=x_1^{(\sigma_1)}\partial_{x_0^{(\sigma_1)}}$ is the sub-ring
    $k[x_0^{(\sigma_1),p},x_1^{(\sigma_1)},
    x_0^{(\sigma_l)}\mid l\geq 2]$
which is regular.
\end{proof}

So for the rest of the article, we make the following hypothesis:
\begin{assumption}
We have $\bD_\mathbf{d}\geq 2$. 
\end{assumption}

\subsection{Weighted blow-up of the ramification locus}\label{section:weighted_blow-up}
We perform a weighted blow-up along the fixed-point scheme of the $\alpha_p$-action, and show that the induced action on the weighted blow-up (understood as a stack) is free.

Let $b\colon \sY=\Bl_\sI\bA^d\to \bA^d$ be the blow-up of the Rees algebra
        $$\mathcal{I}=\sum_{(\sigma,i)\in\mathbf{I}^*}
        \left(x_i^{(\sigma)},i\right).$$
Then $\sY$ is a regular tame Deligne--Mumford stack that is schematic in codimension one. See \cite{Quek_Rydh_Weighted_blowups} for the general theory of weighted blow-ups: for our purpose, what is recalled in \cite[\S 2.4]{Posva_Resolution_1-foliations} will be sufficient. Recall that we assume $\mathbb{D}_\mathbf{d}\geq 2$, so $\mathcal{I}$ is not principal and the blow-up $b$ is non-trivial.

\subsubsection{Irreducible case}
We start with the case of an irreducible representation, i.e.\ $l=1$. For the sake of readability, we momentarily drop the $\sigma$ from our notations, and write $d_1=\delta$. We have $\mathcal{I}=\sum_{i=1}^\delta (x_i,i)$.

We write down equivariant charts for $\sY$. For each $i=1\dots,\delta$, let $C_i=\Spec k[u_0,\dots,u_\delta]$ with a morphism $b_i\colon C_i\to \bA^d$ given by
        $$x_i=u_i^i, \quad 
        x_j=u_i^ju_j \ \ (j\neq 0,i), \quad 
        x_0=u_0,$$
and a $\mu_i$-action given by the diagonal automorphism
            \begin{equation*}
            \chi_i\colon \quad  
            u_i\mapsto \xi_i u_i, \quad 
            u_j\mapsto \xi_i^{-j}u_j \ (j\neq i,0), \quad 
            u_0\mapsto u_0
            \end{equation*}
where $\xi_i$ is a primitive $i$-th root of unity. (We abuse notations by not differentiating coordinates on the different charts: but we will never compare them directly, so this shall not cause any confusion.) 
By definition of weighted blow-ups, we have:

\begin{lemma}
With the above notations, $\bigsqcup_i[C_i/\mu_i]\twoheadrightarrow \sY$ is a Zariski open covering of the stack $\sY$, and $b_i\colon C_i\to \bA^d$ factors as 
        $$C_i\to [C_i/\mu_i] \hookrightarrow \sY \overset{b}{\to} \bA^d$$
where the first morphism is the quotient.
\end{lemma}
 
Now we compute the pullback of $\sF$ to $C_i$. 

\begin{proposition}\label{prop:pullback_derivation_irred_case}
The $1$-foliation $b_i^*\sF$ is regular for every $i$.
\end{proposition}
\begin{proof}
We distinguish three cases.
\begin{enumerate}
    \item On the chart $C_\delta$ we have
                $$b_\delta^*\partial = 
                u_\delta\cdot \underbrace{\left[ 
                u_1\partial_{u_0}+ \sum_{j=2}^{\delta-1}u_j\partial_{u_{j-1}} 
                +\partial_{u_{\delta-1}}
                \right]}_{\psi_\delta}.$$
        The $1$-foliation $b_\delta^*\sF$ is generated by $\psi_\delta$ and therefore it is regular.

    \item On the chart $C_1$ we compute
                \begin{eqnarray*}
                b_1^*\partial &=& 
                b_1^*\left(\sum_{j=2}^{\delta-1} x_{j+1}\partial_{x_{j}}
                +x_2\partial_{x_1}+x_1\partial_{x_0}
                \right)\\
                &=& 
                \sum_{j=2}^{\delta-1} \frac{u_1^{j+1}u_{j+1}}{u_1^j}\partial_{u_j} 
                +\frac{u_1^2u_2}{u_1}\left[ 
                u_1\partial_{u_1} -
                \sum_{j=2}^\delta ju_j\partial_{u_j}
                \right] 
                +u_1\partial_{u_0} \\
                &=&
                u_1 \cdot \psi_1
                \end{eqnarray*}
            where
                $$\psi_1 =
                \partial_{u_0}+u_1u_2\partial_{u_1} 
                +\sum_{j=2}^{\delta-1} (u_{j+1}-ju_2u_j)\partial_{u_j}
                - \delta u_2u_\delta\partial_{u_\delta}
                $$
        which defines a regular $1$-foliation on $C_1$.

    \item For $1<i<\delta$, we compute on the charts $C_i$:
                \begin{eqnarray*}
                    b_i^*\partial &=& 
                    b_i^*\left(\sum_{\substack{j\in \{2,\dots, \delta\} \\ j\neq i, i+1}}
                    x_j \partial_{x_{j-1}}
                    + x_{i+1}\partial_{x_i} +x_i\partial_{x_{i-1}}+x_1\partial_{x_0}
                    \right)\\
                    &=&
                    \sum_{\substack{j\in \{2,\dots, \delta\} \\ j\neq i, i+1}} \frac{u_i^ju_j}{u_i^{j-1}}
                    \partial_{u_{j-1}} 
                    + \frac{u_i^{i+1}u_{i+1}}{i u_i^i}
                    \left[ 
                    u_i\partial_{u_i}-
                    \sum_{\substack{j\in \{1,\dots, \delta\} \\ j\neq i}}
                    ju_j\partial_{u_j}
                    \right] \\
                    && \quad \quad +\frac{u_i^i}{u_i^{i-1}}\partial_{u_{i-1}}
                    +u_1u_i\partial_{u_0} \\
                    &=&
                    \sum_{\substack{j\in \{1,\dots, \delta-1\} \\ j\neq i-1, i}}
                    u_iu_{j+1}\partial_{u_j} 
                    + \frac{1}{i}u_i^2u_{i+1}\partial_{u_i}
                    - \sum_{\substack{j\in \{1,\dots, \delta\} \\ j\neq i}}
                    \frac{j}{i}u_iu_{i+1}u_j\partial_{u_j} \\
                    && \quad \quad +u_i\partial_{u_{i-1}} +u_1u_i\partial_{u_0} \\
                    &=& 
                    u_i \cdot \psi_i
                \end{eqnarray*}
            where
                \begin{eqnarray*}
                \psi_i &=&
                    \sum_{\substack{j\in \{1,\dots, \delta-1\} \\ j\neq i-1, i}}
                    \left( u_{j+1}-\frac{j}{i}u_{i+1}u_j\right)
                    \partial_{u_j}
                    + u_1\partial_{u_0} 
                    -\frac{\delta}{i}u_{i+1}u_\delta\partial_{u_\delta} \\ 
                    &&\quad\quad +  \left(1-\frac{i-1}{i}u_{i+1}u_{i-1}\right)
                    \partial_{u_{i-1}}
                    +\frac{1}{i}u_iu_{i+1}\partial_{u_i}.
                \end{eqnarray*}
            We claim that $\psi_i$ defines a regular $1$-foliation on $C_i$. Indeed, its singular locus $Z$ is defined by the radical of the ideal generated by the coefficients of all the $\partial_{u_j}$ in the above expression. That ideal contains 
            $1-\frac{i-1}{i}u_{i+1}u_{i-1}$, so the functions $u_{i+1}$ and $u_{i-1}$ must be invertible along $Z$. By looking at the coefficient of $\partial_{u_0}$ and of $\partial_{u_\delta}$ we also get $u_1=u_\delta=0$ along $Z$. Now:
                \begin{itemize}
                    \item Suppose $i>2$. Then consider the sum over $j$: since $u_1=0$ along $Z$, by looking at the coefficient of $\partial_{u_1}$ we must have $u_2=0$ along $Z$. Continuing by induction on $j=1,\dots,i-2$, we obtain that $u_{i-1}=0$ along $Z$. So $Z$ must be empty.
                    
                    \item Suppose $i=2$. 
                        \begin{enumerate}
                            \item If $\delta>3$, by looking at the $j=3$ summand in the sum over $j$ we get that $u_4$ is invertible along $Z$. Continuing by induction on $j=3,\dots,\delta-1$, we obtain that $u_\delta$ is invertible $Z$. So $Z$ must be empty.
                            \item If $\delta=3$ then
                                    $$\psi_2=
                                    u_1\partial_{u_0} 
                                    -\frac{3}{2}u_3^2\partial_{u_3}
                                    +(1-\frac{1}{2}u_1u_3)\partial_{u_1}
                                    +\frac{1}{2}u_2u_3\partial_{u_2}
                                    $$
                                and we immediately see that $\psi_2$ defines a regular $1$-foliation.
                        \end{enumerate}
                \end{itemize}
\end{enumerate}
Thus in every case, the pullback $1$-foliation $b_i^*\sF$ is regular.
\end{proof}

\begin{corollary}\label{cor:pullback_foliation_is_regular}
The pullback $1$-foliation $b^*\sF$ is regular on $\sY$. 
\end{corollary}
\begin{proof}
By definition \cite[Definition 3.1.5]{Posva_Resolution_1-foliations}, this follows from the regularity of the $b_i^*\sF$.
\end{proof}

\subsubsection{General case}
We go back to the general case. For each $(\sigma,i)\in\mathbf{I}^*$ we have equivariant charts
        $$C_i^{(\sigma)} =\Spec k\left[ 
        u_j^{(\pi)} \mid 
        (\pi,j)\in\mathbf{I}
        \right]$$
with morphism $b_i^{(\sigma)}\colon C_i^{(\sigma)}\to \bA^d$ given by
        \begin{equation}\label{eqn:structure_map_equiv_charts}
        x_j^{(\sigma)}=\begin{cases}
            u_0^{(\sigma)} & \text{if } j=0,\\
            (u_i^{(\sigma)})^i & \text{if }j=i,\\
            (u_i^{(\sigma)})^ju_j^{(\sigma)} & \text{if }j\notin \{0,i\},
        \end{cases}
        \quad 
        \text{and for }\pi\neq \sigma: \ 
        x_j^{(\pi)}=\begin{cases}
            u_0^{(\pi)} & \text{if }j=0, \\
            (u_i^{(\sigma)})^ju_j^{(\pi)} & \text{if }j\neq 0.
        \end{cases}
        \end{equation}
There is also a $\mu_i$-action on $C_i^{(\sigma)}$ given by
        \begin{equation}\label{eqn:action_on_equiv_chart}
        \chi_i^{(\sigma)}\colon \quad 
        u_j^{(\pi)}\mapsto \begin{cases}
            \xi_i u_i^{(\sigma)} & \text{if }\pi=\sigma \text{ and } j=i,\\
            \xi_i^{-j}u_j^{(\pi)} & \text{otherwise.}
        \end{cases}
        \end{equation}
where $\xi_i$ is a primitive $i$-th root of unity.

\begin{lemma}
With the above notations, $\bigsqcup_{(\sigma,i)\in\mathbf{I}^*}[C_i^{(\sigma)}/\mu_i]\twoheadrightarrow \sY$ is a Zariski open covering of the stack $\sY$, and $b_i^{(\sigma)}\colon C_i^{(\sigma)}\to \bA^d$ factors as 
        $$C_i^{(\sigma)}\to [C_i^{(\sigma)}/\mu_i] \hookrightarrow \sY \overset{b}{\to} \bA^d$$
where the first morphism is the quotient.
\hfill \qedsymbol
\end{lemma}

\begin{lemma}\label{lemma:schematic_in_codim_one}
The stack $\sY$ is schematic in codimension one. 
\end{lemma}
\begin{proof}
It suffices to show that each $[C_i^{(\sigma)}/\mu_i]$ is schematic in codimension one: for this we can assume that $i>1$. Since its non-schematic locus is equal to the image of the ramification locus $R$ of the $\mu_i$-action on $C_i^{(\sigma)}$, it suffices to show that $R$ has codimension at least $2$. We have $R=\bigcup_{s}\Fix(\mu_{i/s})$ where $s$ runs through the positive divisors of $i$ not equal to $i$, and $\Fix(\mu_{i/s})$ is the fixed locus of the action of the non-trivial sub-group $\mu_{i/s}$ of $\mu_i$. The sub-group $\mu_{i/s}$ acts by 
        \begin{equation}\label{eqn:action_subgroup}
        u_j^{(\pi)}\mapsto \begin{cases}
            \xi_i^s u_i^{(\sigma)} & \text{if }\pi=\sigma \text{ and } j=i,\\
            \xi_i^{-js}u_j^{(\pi)} & \text{otherwise.}
        \end{cases}
        \end{equation}
So in every case we have $\Fix(\mu_i)\subseteq V(u_1^{(\sigma)}, u_i^{(\sigma)})$, which shows that $\Fix(\mu_{i/s})$ has codimension $\geq 2$ in $C_i$. This implies that $\codim_{C_i}R\geq 2$ as desired.
\end{proof}

\begin{proposition}\label{prop:pullback_foliations_are_reg_gen_case}
The pullback foliations $b^{(\sigma),*}_i\sF$ and $b^*\sF$ are regular.
\end{proposition}
\begin{proof}
It suffices to show that each $b^{(\sigma),*}_i\sF$ is regular. We check easily that for $\pi\neq \sigma$,
    $$b_i^{(\sigma),*}\partial^{(\pi)} = 
    u_i^{(\sigma)}\cdot 
    \underbrace{\sum_{i=1}^{d_\pi} u_i^{(\pi)}
    \partial_{u_{i-1}^{(\pi)}}}_{\psi^{(\pi)}_i}.$$
Write 
$b_i^{(\sigma),*}\partial^{(\sigma)}=u_i^{(\sigma)}\cdot \psi_i^{(\sigma)}$ where $\psi_i^{(\sigma)}$ is defined as in the proof of \autoref{prop:pullback_derivation_irred_case}. Then
    $$b_i^{(\sigma),*}\partial =u_i^{(\sigma)}\cdot 
    \sum_{\pi=1}^l\psi_i^{(\pi)}.$$
As the coefficients of $\psi_i^{(\sigma)}$ generated the unit ideal, so do the coefficients of $\sum_{\pi=1}^l\psi_i^{(\pi)}$. Hence $b^{(\sigma),*}_i\sF$ is indeed regular.
\end{proof}

\begin{lemma}
The $1$-foliation $b_i^{(\sigma),*}\sF$ is $\mu_i$-equivariant. 
\end{lemma}
\begin{proof}
Let $\mathscr{F}_i$ be the pullback of $\sF$ \emph{as an $\sO_{\bA^d}$-module} along $b_i^{(\sigma)}$. Then $\mathscr{F}_i$ is automatically a $\mu_i$-equivariant module, since $b_i^{(\sigma)}$ is $\mu_i$-invariant. As we saw in the proof of \autoref{prop:pullback_foliations_are_reg_gen_case}, the module $\mathscr{F}_i$ is generated by $b_i^{(\sigma),*}\partial =u_i^{(\sigma)}\cdot \sum_{\pi}\psi_i^{(\pi)}$, while the $1$-foliation $b_i^{(\sigma),*}\sF$ is generated by $\sum_{\pi}\psi_i^{(\pi)}$. Since $u_i^{(\sigma)}$ is a $\mu_i$-eigenfunction, it follows that $b_i^{(\sigma),*}\sF$ is preserved by the action of $\mu_i$.
\end{proof}


\begin{corollary}\label{cor:action_descends_inf_qt}
The $\mu_i$-action on $C_i^{(\sigma)}$ descends to $Z_i=C_i^{(\sigma)}/b_i^{(\sigma),*}\sF$, and the quotient morphism
$q_i^{(\sigma)}\colon C_i^{(\sigma)}\to Z_i^{(\sigma)}$ is $\mu_i$-equivariant.
\end{corollary}
\begin{proof}
The elements of $\Gamma\left(Z_i^{(\sigma)}, \sO_{Z_i^{(\sigma)}}\right)$ are the elements of $\Gamma\left(C_i^{(\sigma)}, \sO_{C_i^{(\sigma)}}\right)$ which are annihilated by $\sum_\pi \psi_i^{(\pi)}$, and 
the previous lemma shows that this property is preserved by applying $\chi_i^{(\sigma)}$. The statement follows.
\end{proof}

At this point we introduce some additional notations.
\begin{notation}\label{notation:overall_situation}
We let:
    \begin{itemize}
    \item $Z_i^{(\sigma)}=C_i^{(\sigma)}/b_i^{(\sigma),*}\sF$ with quotient morphism $q_i^{(\sigma)}\colon C_i^{(\sigma)}\to Z_i^{(\sigma)}$ (as already introduced),
    \item $W_i^{(\sigma)}=Z_i^{(\sigma)}/\mu_i$ with quotient morphism $t_i^{(\sigma)}\colon Z_i^{(\sigma)}\to W_i^{(\sigma)}$,
    \item $\sW=\sY/b^*\sF$ in the sense of \cite[\S 3.2]{Posva_Resolution_1-foliations}, and
    \item $c_\sY\colon \sY\to Y$ and $c_\sW\colon \sW\to W$ are the coarse moduli spaces. In fact, $\sY$ is tame by construction and $\sW$ is tame by \cite[Lemma 3.3.2]{Posva_Resolution_1-foliations}; so the coarse moduli spaces exist, and are good moduli spaces.
\end{itemize}
\end{notation}

The following lemma spells out the natural maps between these objects and introduces more notations.

\begin{lemma}\label{lemma:overall_notations}
With the notations as above:
    \begin{enumerate}
        \item each $Z_i^{(\sigma)}$ is a regular affine variety;
        \item $\sW$ is a regular tame Deligne--Mumford stack that is schematic in codimension one;
        \item there is a commutative diagram
                $$\begin{tikzcd}
                    && \bigsqcup_{(\sigma,i)\in \mathbf{I}^*} Z_i^{(\sigma)} \arrow[rr, "\bigsqcup t_i^{(\sigma)}"] \arrow[dd, two heads, "\alpha"] && 
                    \bigsqcup_{(\sigma,i)\in \mathbf{I}^*} W_i^{(\sigma)} \arrow[dd, two heads, "\beta"] \\
                    \bigsqcup_{(\sigma,i)\in \mathbf{I}^*} C_i^{(\sigma)} \arrow[urr, "\bigsqcup q_i^{(\sigma)}"] \arrow[dd, two heads] \\
                    && \sW \arrow[rr, "c_\sW"] && W \arrow[d, "g"] \\
                    \sY \arrow[rr, "c_\sY"] \arrow[urr] \arrow[rrd, bend right, "b" below left] && Y \arrow[urr, "q" below right] \arrow[d, "f"]
                    && X\\
                    && \bA^d \arrow[rru] \\
                \end{tikzcd}$$
            where $\alpha$ is an \'{e}tale covering, and $\beta$ is a Zariski open covering;
        \item $Y$ is normal, and $f$ is birational with a unique exceptional divisor $E$;
        \item $W$ is normal with only tame toroidal singularities, $q$ is the quotient morphism $Y\to Y/f^*\sF=W$, and $g$ is birational with a unique exceptional divisor $\widetilde{E}=q(E)$.
        \item We denote by $\mathcal{E}$, resp.\ by $\widetilde{\mathcal{E}}$, the unique prime divisor of $\sY$, resp.\ of $\sW$, which maps to $E$, resp.\ to $\widetilde{E}$.
    \end{enumerate}
\end{lemma}
\begin{proof}
That $\sW$ is a tame stack was mentioned above. Combining \autoref{lemma:schematic_in_codim_one} and \autoref{prop:qt_is_stab_preserving}, we find that $\sW$ is schematic in codimension one. Since the $C_i^{(\sigma)}$s provide affine open charts for $\sY$, the $Z_i^{(\sigma)}$s provide affine open charts for $\sW$ (see \cite[\S 3.2]{Posva_Resolution_1-foliations}). The latter charts are regular by \autoref{prop:pullback_foliations_are_reg_gen_case} and \cite[Lemma 2.5.10]{Posva_Singularities_quotients_by_1-foliations}. As $\sW$ has an open covering by the $[Z_i^{(\sigma)}/\mu_i]$, it follows that $\sW$ is regular and Deligne--Mumford. It follows that $W$ has only normal tame toroidal singularities.

By the universal property of the coarse moduli, the morphism $b\colon \sY\to \bA^d$ factors through a morphism $f\colon Y\to X$. Since $b$ is birational, so is $f$. In fact, $f$ is the \emph{schematic} weighted blow-up of $\sI$: since its center is irreducible and smooth, it has a unique exceptional divisor $\widetilde{E}$.

The existence of $q$ also follows from the universal property of $c_\sY$. That $W=Y/f^*\sF$ and $q$ is the quotient morphism follows from \cite[Proof of Lemma 3.3.2]{Posva_Resolution_1-foliations}. The morphism $g\colon W\to X$ is constructed as follows: since $Y\to W$ and $\bA^d\to X$ are homeomorphisms, it suffices to show that the composition $\sO_{\bA^d}^\sF\hookrightarrow \sO_{\bA^d}\to f_*\sO_Y$ factors through $f_*\sO_{Y}^{f^*\sF}=f_*\sO_W$, which is clear. Then $g$ is birational, with unique exceptional divisor $q(E)$. 
\end{proof}

\begin{corollary}
    The variety $X$ admits a resolution of singularities.
\end{corollary}
\begin{proof}
We apply Bergh's destackification \cite{Bergh_Functorial_destackification} to $\sW$ (which has only abelian stabilizers, by \autoref{prop:qt_is_stab_preserving} and the fact that $\sY$ has only abelian stabilizers) to obtain a projective birational morphism of DM stacks $\widetilde{\sW}\to \sW$ such that the coarse moduli space $\widetilde{W}$ of $\widetilde{\sW}$ is regular: then $\widetilde{W}$ is a resolution of $X$.
\end{proof}

We regard $\sW\to X$ as a stacky resolution of singularities, and $W\to X$ as a partial resolution of singularities. We will take advantage of the simple local structure of $\sW$ and $W$ to study the singularities of $X$.

\subsection{MMP singularities of linear $\alpha_p$-quotients}
In this sub-section we use $\sW$ and $W$ to discuss the MMP singularities of the quotient $X=\bA^d/(\alpha_p,\mathbf{d})$. To begin with, we compute some discrepancies along the morphisms $f$ and $g$.

\begin{lemma}\label{lemma:discrep_along_weighted_blow_up}
The $f$-exceptional divisor $E$ is $f^*\sF$-invariant, $a(E;\sF)=-1$ and $a(E;\bA^d)=\bD_\mathbf{d}-1$.
\end{lemma}
\begin{proof}
Everything can be checked generically on $E$. Let us look on the chart $C_1^{(\sigma)}$ of $\sY$ (for any $\sigma$ with $d_\sigma>0$), which is an open sub-scheme of $Y$. The divisor $E\cap C_1^{(\sigma)}$ is defined by the function $u_1^{(\sigma)}$, and
        $$b_1^{(\sigma),*}\partial=u_1^{(\sigma)}\cdot \sum_\pi \psi_1^{(\pi)}.$$
As $\sum_\pi \psi_1^{(\pi)}$ generates $b_1^{(\sigma),*}\sF$, this shows $a(E;\sF)=-1$. As 
        $$\left(\sum_\pi \psi_1^{(\pi)}\right)(u_1^{\sigma)})=
        \psi_1^{(\sigma)}(u_1^{(\sigma)})
        =u_1^{(\sigma)}u_2^{(\sigma)}$$
we see that $E$ is $f^*\sF$-invariant. Finally, using \autoref{eqn:structure_map_equiv_charts} we have
        $$f^*\left(
        \bigwedge_\pi \bigwedge_{i=0}^{d_\pi} dx_i^{(\pi)}
        \right)
        =\left(u_1^{(\sigma)}\right)^{-1+\sum_\pi \sum_{i=1}^{d_\pi}i}\cdot 
        \bigwedge_\pi \bigwedge_{i=0}^{d_\pi} du_i^{(\pi)}$$
(observe that $f^*x_1^{(\pi)}=u_1^{(\sigma)}u_1^{(\pi)}$, except when $\pi=\sigma$). So we find
        $$a(E;\bA^d)=-1+\sum_\pi \sum_{i=1}^{d_\pi}i 
        = -1+\sum_\pi \frac{d_\pi(d_\pi+1)}{2}
        =-1+\bD_\mathbf{d}$$
as claimed.
\end{proof} 

\begin{corollary}
We have $a(\widetilde{E};X)=\bD_\mathbf{d}-p$.
\end{corollary}
\begin{proof}
This follows from \autoref{lemma:discrep_along_weighted_blow_up} and the scheme-theoretic case of \autoref{thm:discrep_infinit_qt_stack} (see \cite[(4.2.6.f)]{Posva_Singularities_quotients_by_1-foliations}).
\end{proof}

We can now prove the main result of this sub-section.

\begin{theorem}\label{thm:MMP_sing_alpha_p_qt}
Assume that $\mathbb{D}_\mathbf{d}\geq 2$. Then the variety $X$ is lc (resp.\ canonical, resp.\ terminal) if and only if
$\bD_\mathbf{d}\geq p-1$ 
(resp.\ $\bD_\mathbf{d}\geq p$, resp.\ $\bD_\mathbf{d}>p$).
\end{theorem}
\begin{proof}
Write $\gamma=a(\widetilde{E}; X)=\bD_\mathbf{d}-p$. The \emph{only if} direction is immediate, and it remains to prove sufficiency.

Let $m\colon W'\to W$ be a birational proper morphism from a normal variety, and let $\widetilde{D}$ be an $m$-exceptional divisor. We have to bound from below the discrepancy $a(\widetilde{D}; W,-\gamma \widetilde{E})$: we will do so by studying of the stacky pair $(\sW,-\gamma\widetilde{\mathcal{E}})$ and using the tools introduced in \autoref{section:DM_stacks} and \autoref{section:foliations}.

We start by introducing some notations. As in \autoref{prop:birat_model_of_stack}, let $\sW'$ be the normalization of of $\sW\times_W W'$: it is endowed with a birational proper representable morphism $h\colon \sW'\to \sW$, and the natural morphism $\sW'\to W'$ is the coarse moduli morphism. We let $\sY'$ be the normalization of $\sY\times_\sW\sW'$: the natural morphism $h'\colon \sY'\to \sY$ is birational, proper and representable.
Let $\widetilde{\mathcal{D}}$ be the unique $h$-exceptional prime divisor of $\sW'$ mapping to $\widetilde{D}$ in $W'$. Let also $\mathcal{D}$ be the unique prime divisor of $\sY'$ mapping to $\widetilde{\mathcal{D}}$: then $\mathcal{D}$ is $h'$-exceptional. We sum up these notations with the commutative diagram
        $$\begin{tikzcd}
            \sD \arrow[d, hook] & \widetilde{\sD}\arrow[d, hook]
            & \widetilde{D}\arrow[d, hook] \\
            \sY' \arrow[r, "q'"] \arrow[d, "h'"] & \sW' \arrow[d, "h"] \arrow[r] & W' \arrow[d, "m"] \\
            \sY \arrow[r] & \sW \arrow[r] & W.
        \end{tikzcd}$$
By Jacobson's correspondence \cite[Theorem 3.3.4]{Posva_Resolution_1-foliations}, the morphism $q'$ is the quotient by the $1$-foliation $h'^*\sF_\sY$ where $\sF_\sY=b^*\sF$. Moreover, by \autoref{lemma:discrep_along_weighted_blow_up} the divisor $\mathcal{E}$ on $\sY$ is $\sF_\sY$-invariant: so $\mathcal{E}_{\sY/\sF_\sY}=\widetilde{\mathcal{E}}$, with the notations of \autoref{eqn:natural_divisor_on_qt}. By applying \autoref{thm:discrep_infinit_qt_stack} we obtain
    \begin{equation}\label{eqn:computation_stacky_discrep}
        a(\widetilde{\mathcal{D}}; \sW, -\gamma\widetilde{\mathcal{E}})
        =
        p^{-\epsilon_{\sY'}(\mathcal{D})} \cdot 
        \left[
        a(\mathcal{D};\sY,-\gamma\mathcal{E})+
        (p-1)\cdot a(\mathcal{D};\sF_\sY)
        \right].
    \end{equation}  
Since $\sY$ and $\sF_\sY$ are regular, by \autoref{prop:reg_fol_is_can} we have $a(\mathcal{D};\sF_\sY)\geq 0$. An \'{e}tale covering of the sub-pair $(\sY,-\gamma\mathcal{E})$ is given by the sub-pairs $(C_i^{(\sigma)}, -\gamma E^{(\sigma)}_i)$ where $E^{(\sigma)}_i$ is the pullback of $E$ to $C^{(\sigma)}_i$. The support of $E_i^{(\sigma)}$ is the zero locus of the function $u_i^{(\sigma)}$, hence $C^{(\sigma)}_i$ and $\Supp(E^{(\sigma)}_i)$ are regular. So it follows from \cite[Corollary 2.11]{Kollar_Singularities_of_the_minimal_model_program} and \autoref{eqn:discrep_divisor_stacks} that $a(\mathcal{D};\sY,-\gamma \mathcal{E})\geq 0$ when $\gamma\geq -1$, with strict inequality when $\gamma\geq 0$.
So \autoref{eqn:computation_stacky_discrep} implies that 
    \begin{equation}\label{eqn:estimation_stacky_discrep}
        a(\widetilde{\mathcal{D}}; \sW, -\gamma\widetilde{\mathcal{E}}) \geq 0
        \quad \text{when }\gamma\geq -1,
        \text{ with strict inequality when }
        \gamma\geq 0.
    \end{equation}
    
Now we get back to the discrepancies of $(W,-\gamma \widetilde{E})$.
Since $\sW$ is schematic in codimension one (\autoref{lemma:overall_notations}), the divisor $K_W-\gamma\widetilde{E}$ on $W$ pullbacks to $K_\sW-\gamma\widetilde{\sE}$ on $\sW$.
By \autoref{eqn:discrep_of_stacks_and_mod} we obtain
        \begin{equation}\label{eqn:schematic_vs_stacky_discrep}
        a(\widetilde{D};W,-\gamma \widetilde{E})
        =\frac{a(\widetilde{\mathcal{D}}; \sW,-\gamma \widetilde{\mathcal{E}})+1}
        {r_{\sW'}(\widetilde{\mathcal{D}})}-1.
        \end{equation}
Assume that $\gamma\geq -1$: from \autoref{eqn:estimation_stacky_discrep} and the above equality we obtain
        $$a(\widetilde{D};W,-\gamma \widetilde{E})
        \geq\frac{1}
        {r_{\sW'}(\widetilde{\mathcal{D}})}-1.$$
Therefore we obtain $a(\widetilde{D};W,-\gamma \widetilde{E})> -1$ when $\gamma\geq -1$. So when $\gamma=-1$ the variety $X$ is lc. 

This happens to prove the canonical case as well, for $X$ is $1$-Gorenstein \cite[Lemma 4.2]{Tonini_Yasuda_Motivic_McKay_cor_for_alpha_p} and so the discrepancies $a(\widetilde{D};W,-\gamma \widetilde{E})$ are integers. Alternatively, the following argument will prove the canonical case and the terminal case at the same time.

Assume that $\gamma\geq 0$: we will show that $a(\widetilde{D};W,-\gamma \widetilde{E})\geq 0$, with strict inequality when $\gamma \geq 1$. This will complete the proof. This is clear if $r_{\sW'}(\widetilde{\mathcal{D}})=1$ as
        $$a(\widetilde{D};W,-\gamma \widetilde{E})=a(\widetilde{\mathcal{D}};\sW,-\gamma \widetilde{\sE}) >0$$
thanks to \autoref{eqn:estimation_stacky_discrep}. 
So assume that $r_{\sW'}(\widetilde{\mathcal{D}})>1$. We have
        \begin{eqnarray*}
        a(\widetilde{\sD};\sW,-\gamma \widetilde{\sE})
        &=& a(\widetilde{\sD};\sW)
        +\gamma\cdot \mult_{\widetilde{\sD}}h^*\widetilde{\sE} \\
        &\geq & \codim_\sW h(\widetilde{\sD})-1
        +\gamma
        \end{eqnarray*}
where the inequality follows from \autoref{eqn:discrep_divisor_stacks}, \cite[Lemma 2.1.2]{Posva_Pathological_MMP_sing} and the fact that $\sW$ is regular. In view of \autoref{eqn:schematic_vs_stacky_discrep}, we are going to compare the codimension of $h(\widetilde{\sD})$ with $r_{\sW'}(\widetilde{\sD})$.

To do so, observe that 
$\codim_{\sW}h(\widetilde{\sD})=\codim_\sY h'(\sD)$,
and that $r_{\sW'}(\widetilde{\sD})=r_{\sY'}(\sD)$ by \autoref{prop:qt_is_stab_preserving}. In particular $r_{\sY'}(\mathcal{D})>1$. Say that the generic point $\xi$ of $\sD$ is mapped by $h'$ to the open sub-stack $[C^{(\sigma)}_i/\mu_i]$ of $\sY$. As $h'$ is representable, $h'(\xi)$ must belong to the locus $\mathcal{Z}$ with stabilizer $\mu_t<\mu_i$ for some $t>1$ with $i=st$. This implies $r_{\sY'}(\sD)\leq t$.
By \autoref{eqn:action_on_equiv_chart} we see that $\sZ$ is the image in $[C^{(\sigma)}_i/\mu_i]$ of the locus $Z$ of $C^{(\sigma)}_i$ where the automorphism $(\chi_i^{(\sigma)})^{\circ s}$ acts trivially. By looking at the description of $\chi^{(\sigma)}_i$, we see that $u_1^{(\sigma)}, \dots, u_{t-1}^{(\sigma)}$ and $u_i^{(\sigma)}$ must vanish along $Z$. In particular,
        $$\codim_\sY h'(\mathcal{D})
        \geq \codim_{[C^{(\sigma)}_i/\mu_i]}\sZ
        =\codim_{C^{(\sigma)}_i}Z
        \geq t.$$
Putting these estimates together, we get $\codim_\sW h(\widetilde{\sD})\geq r_{\sW'}(\widetilde{\sD})$.
Therefore
        \begin{eqnarray*} 
        a(\widetilde{D};W,-\gamma \widetilde{E})
        &=& 
        \frac{a(\widetilde{\mathcal{D}}; \sW,-\gamma \widetilde{\mathcal{E}})+1}
        {r_{\sW'}(\widetilde{\mathcal{D}})}-1 \\
        &\geq & 
        \frac{\codim_\sW h(\widetilde{\sE})+\gamma}{r_{\sW'}(\widetilde{\sD})}-1 \\
        &\geq & \frac{\gamma}{r_{\sW'}(\widetilde{\sD})}.
        \end{eqnarray*}
So we have obtained $a(\widetilde{D};W,-\gamma \widetilde{E})\geq 0$, with strict inequality when $\gamma\geq 1$. The theorem is thus proved.
\end{proof}

\begin{remark}\label{rmk:qt_not_CM}
It follows from \cite[Lemma 2.5.4]{Posva_Singularities_quotients_by_1-foliations} that $X$ is never $F$-injective, and if the fixed-point scheme of the $\alpha_p$-action on $X$ has codimension $\geq 3$ (which happens as soon as $\bD_\mathbf{d}\geq 3$) then $X$ is not Cohen--Macaulay.
\end{remark}

\subsection{Stratification of the partial resolution}\label{section:singularities_partial_resolution}
In this sub-section we stratify the toroidal singularities of the partial resolution $W$ of $X$. It follows from \autoref{lemma:sing_locus_of_qt} and \cite[Lemma 2.5.10]{Posva_Singularities_quotients_by_1-foliations} that the singular locus of $X$ is the image of the center of the blow-up $f$, in other words $g(\widetilde{E})$. So $W\setminus \widetilde{E}$ is regular, and we will study the singularities of $W$ along $\widetilde{E}$.

We start by describing the exceptional divisor of $f\colon Y\to \bA^d$. Recall that we have commutative diagrams
        $$\begin{tikzcd}
        C_i \arrow[rr, bend left, "r^{(\sigma)}_i" above] \arrow[r] &
        {[}C_i^{(\sigma)}/\mu_i{]} \arrow[d, hook] \arrow[r] & C_i^{(\sigma)}/\mu_i \arrow[d, hook] \\
        &\sY \arrow[r, "c_\sY"] & Y
        \end{tikzcd}$$
where the vertical arrows are open embeddings, and $r^{(\sigma)}_i$ is the quotient by the action of $\mu_i$.
We let 
$\bP(\mathbf{d})$ be the schematic weighted projective space
        $$\bP(\mathbf{d})=
        \Proj_k k\left[ 
            U_i^{(\sigma)} \mid 
            (\sigma,i)\in \mathbf{I}^*
        \right]$$
where the homogeneous coordinate $U_i^{(\sigma)}$ has degree $i$. Notice that we do not include the index $i=0$.

\begin{lemma}
The $f$-exceptional divisor 
$E\cong \bP(\mathbf{d})\times \bA^l$ where the coordinates on $\bA^l$ are the pullbacks of the coordinates $x_0^{(\sigma)}$, $\sigma=1,\dots,l$, from the base $\bA^d$.
\end{lemma}
\begin{proof}
While $\sY$ is the stack-theoretic weighted blow-up of 
$\sum_\sigma \sum_{i=1}^{d_\sigma}(x_i^{(\sigma)},i)$, its coarse moduli space $Y$ is the scheme-theoretic weighted blow-up of the same Rees algebra. The description of the exceptional divisor is then clear.
\end{proof}

We stratify $E$ as follows. For each homogeneous coordinate $U_i^{(\sigma)}$, let $H_i^{(\sigma)}=\{U_i^{(\sigma)}=0\}\times \bA^l\subset E$. For a subset $\mathscr{S}$ of the set of index pairs $\mathbf{I}^*=\{(\pi,j)\mid \pi=1,\dots,l \text{ and }j=1,\dots,d_\pi\}$, we write
        $$P_{[\mathscr{S}]}=\bigcap_{(\sigma,i)\notin \mathscr{S}}H_i^{(\sigma)}
        \quad \text{and} \quad 
        P_{[\mathscr{S}]}^\circ = P_{[\mathscr{S}]}\setminus \bigcup_{(\pi,j)\in \mathscr{S}}H_j^{(\pi)}.$$
By convention $P_{[\emptyset]}=\emptyset$, so we will assume $\mathscr{S}$ to be non-empty in what follows, and $P_{[\mathbf{I}^*]}=E$. By \autoref{eqn:structure_map_equiv_charts} we have
        $$E\cap C_i^{(\sigma)}/\mu_i
        = r^{(\sigma)}_i\left\{
        u_i^{(\sigma)}=0
        \right\},$$
and this closed subset of $C_i^{(\sigma)}$ is an affine space with coordinates $\{\bar{u}_j^{(\pi)}\mid (\pi,j)\neq (\sigma,i)\}$. It has an action of $\mu_i$ given by the restriction of \autoref{eqn:action_on_equiv_chart}, namely
        $$\bar{\chi}_i^{(\sigma)}\colon \quad \bar{u}_j^{(\pi)}\mapsto \xi_i^{-j}\bar{u}_j^{(\pi)}$$
and the geometric quotient is the open subset $D_+(U_i^{(\sigma)})\times \bA^l$ of $E$, which is the complement of $H^{(\sigma)}_i$. Hence
        $$r^{(\sigma)}_i\left\{
        u_i^{(\sigma)}=0
        \right\}=E\setminus H^{(\sigma)}_i.$$
The hyperplanes $\{u_j^{(\pi)}=0\}$ of $C_i^{(\sigma)}$ are $\mu_i$-invariant, so 
    $$r^{(\sigma),-1}_ir^{(\sigma)}_i\left\{u_j^{(\pi)}=0\right\}
    =\left\{u_j^{(\pi)}=0\right\}$$
and furthermore
    $$r^{(\sigma)}_i\left\{ u_j^{(\pi)}=0 \right\}\cap E=
    H_j^{(\pi)}\setminus H^{(\sigma)}_i.$$
These considerations lead to the following statement:

\begin{lemma}\label{lemma:equiv_description_stratification}
Assume that $\mathscr{S}\subseteq \mathbf{I}^*$ is non-empty, and take $(\sigma,i)\in \mathscr{S}$. Then $P_{[\mathscr{S}]}^\circ$ is contained in the open chart $C_i^{(\sigma)}/\mu_i$, and is equal to the image of the set
        $$\left(
        \bigcap_{(\pi,j)\notin \mathscr{S}}\left\{u^{(\pi)}_j=0\right\}
        \right)
        \setminus 
        \left(
        \bigcup_{(\lambda,e)\in \mathscr{S}} \left\{u_e^{(\lambda)}=0\right\}
        \right)$$
through the quotient morphism $r^{(\sigma)}_i\colon C^{(\sigma)}_i\to C_i^{(\sigma)}/\mu_i$.
\hfill \qedsymbol
\end{lemma}

We introduce the following notations.
\begin{notation}
Fix a non-empty subset of index pairs $\mathscr{S}\subseteq \mathbf{I}^*$.
    \begin{enumerate}
        \item We let $Q_\mathscr{S}^\circ=q(P^\circ_{[\mathscr{S}]})$.

        \item  We write $\gcd(\mathscr{S})=\gcd\{i\mid (\sigma,i)\in \mathscr{S}
\text{ for some }\sigma\}$.
    \end{enumerate}
\end{notation}
\noindent Recall that $q\colon Y\to W$ is a universal homeomorphism. In particular the collection $\{Q_\mathscr{S}^\circ\mid \mathscr{S}\subseteq \mathbf{I}^*\}$ is a decomposition of $\widetilde{E}$ into disjoint locally closed subsets.
We are going to show that the singularity type is constant along each stratum $Q_\mathscr{S}^\circ$.

\begin{proposition}\label{prop:sing_partial_resol}
With the notations as above, let $w\in Q^\circ_\mathscr{S}$ be a closed point. If $\gcd(\mathscr{S})=1$ then $w\in W$ is a regular point. Otherwise the pair $(w\in W,\widetilde{E})$ is a cyclic quotient singularity of type
                $$\frac{1}{\gcd(\mathscr{S})}\left( 
                v_j^{(\pi)}
                \mid \pi=1,\dots,l \text{ and }
                j=0,\dots,d_\pi
                \right)$$
where 
            $$v_j^{(\pi)}=
            \begin{cases}
            p & \text{if }(\pi,j)=(\sigma,i-1),\\
            1 & \text{if }(\pi,j)=(\sigma,i),\\
            -j & \text{otherwise}
            \end{cases}$$
and where we regard $\widetilde{E}$ as corresponding to the index $(\sigma,i)$ in the sum over all $(\pi,j)$
(cf.\ \autoref{def:cyclic_qt}).
\end{proposition}
\begin{proof}
Suppose that $(\sigma,i)\in \mathscr{S}$ and consider the commutative diagram
        $$\begin{tikzcd}
            C_j^{(\pi)} \arrow[r, "q_j^{(\pi)}"] \arrow[d, "r_j^{(\pi)}"] &
            Z_j^{(\pi)} \arrow[d, "t_j^{(\pi)}"] \\
            C_j^{(\pi)}/\mu_j \arrow[r, "q"] & 
            W_j^{(\pi)}
        \end{tikzcd}$$
where we follow \autoref{notation:overall_situation}.
Then, by \autoref{lemma:equiv_description_stratification}, the point $w$ belongs to the image in $W_i^{(\sigma)}$ of the locally closed subset $Z(\mathscr{S})$ of $C_i^{(\sigma)}$ where the coordinate functions $u_j^{(\pi)}$ are invertible if $(\pi,j)\in \mathscr{S}$, and vanish identically otherwise. 

By \autoref{eqn:action_on_equiv_chart} we see that $Z(\mathscr{S})$ is $\mu_i$-invariant. Suppose that the sub-group $\mu_{t}$, with $t>1$, has a fixed point contained in $Z(\mathscr{S})$. Its action is given by \autoref{eqn:action_subgroup}, with $s=i/t$: so in order to have a fixed point, we see that for every $(\pi,j)\in \mathscr{S}$ we must have $\xi_i^{-ij/t}= 1$, as $u_j^{(\pi)}$ is invertible along $Z(\mathscr{S})$. So $t$ must divide each such $j$, and it follows that $t$ divides $\gcd(\mathscr{S})$ (\footnote{
The $\gcd$ has the following property: given a set of positive integers $\mathbf{S}$, if $n$ divides every element of $\mathbf{S}$, then $n$ also divides $\gcd(\mathbf{S})$. This is easily proved by considering decompositions into products of primes.
}). So we obtain that if some $\mu_t$ has a fixed point on $Z(\mathscr{S})$, it is a sub-group of $\mu_{\gcd(\mathscr{S})}$; moreover, the latter acts trivially along $Z(\mathscr{S})$. This shows that the inertia of the $\mu_i$-action on $Z(\mathscr{S})$ is constant and equal to $\mu_{\gcd(\mathscr{S})}$. By \autoref{prop:qt_is_stab_preserving_II} it follows that the inertia of the $\mu_i$-action is constant along $q_i^{(\sigma)}(Z(\mathscr{S}))$ and equal to $\mu_{\gcd(\mathscr{S})}$. It remains to find the weights of this action. 
    
If $\gcd(\mathscr{S})=1$ there is noting to show, so let us assume that $\gcd(\mathscr{S})>1$. Let $s$ be such that $s\cdot \gcd(\mathscr{S})=i$. For simplicity, let $\chi=\chi_i^{(\sigma)}$ be the generator of the $\mu_i$-action \autoref{eqn:action_on_equiv_chart}, and let $\varsigma$ be the automorphism of $Z_i^{(\sigma)}$ satisfying $q_i^{(\sigma)}\circ \chi=\varsigma \circ q_i^{(\sigma)}$ (cf.\ \autoref{cor:action_descends_inf_qt}).

Fix a closed point $z=q_i^{(\sigma)}(x)\in q_i^{(\sigma)}(Z(\mathscr{S}))$ mapping to $w$. We have
        $$\widehat{\sO}_{W,w}\cong \left(
                \widehat{\sO}_{Z_i^{(\sigma)},z}
            \right)^{\langle \varsigma^s \rangle}.$$
Here $\varsigma^s$ gives a $\mu_{\gcd(\mathscr{S})}$-action on the regular local ring $\widehat{\sO}_{Z_i^{(\sigma)}}$. If $\fm$ is its maximal ideal, we need to determine the multi-set of eigenvalues of the induced endomorphism $\overline{\varsigma^s}$ of $\fm/\fm^2$ (that is, the set of eigenvalues counted with multiplicities).

Let us start by determining the eigenvalues of the endomorphism $\overline{\chi^s}$ of $\fn/\fn^2$, where $\fn$ is the maximal ideal of $\sO_{C_i^{(\sigma)},x}$. For each $(\pi,j)$, let $u_j^{(\pi)}(x)\in k$ be the 
representative of the image of $u_j^{(\pi)}$ in the residue field at $x$. Then the 
        $$\widetilde{u}_j^{(\pi)}=u_j^{(\pi)}-u_j^{(\pi)}(x)$$
form a set of regular parameters at $x$. Observe that $u_j^{(\pi)}(x)\neq 0$ if and only if $(\pi,j)\notin \mathscr{S}$, in which case $\chi^s$ acts trivially on $u_j^{(\pi)}$. Let $\xi_i$ be the $i$-th primitive root of unity appearing in \autoref{eqn:action_on_equiv_chart}: then $\zeta=\xi_i^s$ is a primitive $\gcd(\mathscr{S})$-root of unity, and we have
            $$\chi^s\left(\widetilde{u}_j^{(\pi)}\right) =
            \zeta^{\varrho_j^{(\pi)}} \widetilde{u}_j^{(\pi)}
            \quad \text{where} \quad 
            \varrho_j^{(\pi)}=
            \begin{cases}
                1 & \text{if }(\pi,j)=(\sigma,i),\\
                -j & \text{otherwise.}
            \end{cases}$$
Now, as $(\sigma,i)\in \mathscr{S}$ and $\gcd(\mathscr{S})>1$, the function $u_{i-1}^{(\sigma)}$ must vanish along $Z(\mathscr{S})$, and so it belongs to the maximal ideal $\fn$. From the computations in \autoref{prop:pullback_derivation_irred_case} we see that $\psi_i^{(\sigma)}(u_{i-1}^{(\sigma)})$ is invertible in $\sO_{C_i^{(\sigma)},x}$. By a result of Seshadri--Yuan \cite[Lemma 2.4.5]{Posva_Singularities_quotients_by_1-foliations} it follows that $u_{i-1}^{(\sigma),p}$ is a local parameter of $\sO_{Z_i^{(\sigma)},z}$ and that
            \begin{equation*}
                \sO_{C_i^{(\sigma)},x}=
                \sO_{Z_i^{(\sigma)},z}\left[ 
                u_{i-1}^{(\sigma)}
                \right].
            \end{equation*}
It follows that we have an exact sequence of $k$-vector spaces
        \begin{equation}\label{eqn:exact_seq_vct_spaces}
        0\to k\cdot \overline{u_{i-1}^{(\sigma),p}} \to 
        \fm/\fm^2 \overset{\iota}{\longrightarrow} \fn/\fn^2 \to 
        k\cdot \overline{u_{i-1}^{(\sigma)}}
        \to 0.
        \end{equation}
As the inclusion 
$\widehat{\sO}_{Z_i^{(\sigma)},z}\hookrightarrow
\widehat{\sO}_{C_i^{(\sigma)},x}$ is $\mu_{\gcd(\mathscr{S})}$-equivariant (see \autoref{prop:qt_is_stab_preserving_II}), the $k$-linear map $\iota$ appearing in the exact sequence verifies
        $$\overline{\chi}^s\circ \iota=\iota\circ \overline{\varsigma}^s.$$
Using this equality we deduce that the multi-set of eigenvalues of $\overline{\varsigma^s}$ is given by 
        $$\left\{\left\{\zeta^{v_j^{(\pi)}}
        \mid \pi=1,\dots,l \text{ and }j=0,\dots,d_\pi
        \right\}\right\}
        \quad \text{where}\quad 
        v_j^{(\pi)}=
        \begin{cases}
            p & \text{if }(\pi,j)=(\sigma,i-1),\\
            \varrho_j^{(\pi)} & \text{otherwise}.
        \end{cases}$$
To conclude the proof, consider the divisor $t^{(\sigma),-1}_i(\widetilde{E})$. Tts ideal $\mathfrak{a}$ in $\sO_{Z_i^{(\sigma)},z}$ is principal since $Z_i^{(\sigma)}$ is regular. Set-theoretically, its extension $\mathfrak{a}\sO_{C_i{(\sigma)},x}$ cuts outs $\{u_i^{(\sigma)}=0\}$. Since that latter divisor is $b_i^{(\sigma),*}\sF$-invariant (cf.\ the computations in the proof of \autoref{prop:pullback_derivation_irred_case}), it follows from \cite[Lemma 4.2.2]{Posva_Singularities_quotients_by_1-foliations} that $\mathfrak{a}\sO_{C_i{(\sigma)},x}$ is generated by $u_i^{(\sigma)}$. Thus $\mathfrak{a}$ is generated by some $\eta u_i^{(\sigma)}$ where $\eta\in \sO_{Z_i^{(\sigma)},z}^\times$. Let $\eta(z)\in k$ be the representative the image of $\eta$ in the residue field of $z$: then in $\fm/\fm^2$ we have
        $$\overline{\varsigma^s}\left(
        \overline{\eta u_i^{(\sigma)}}\right)
        =\overline{\varsigma^s}\left( \eta_0\cdot  \overline{u_i^{(\sigma)}} \right)
        =
        \eta_0 \cdot \overline{\chi}^s
        \left(\overline{u_i^{(\sigma)}}\right)
        =\zeta^{v_i^{(\sigma)}}\eta_0\cdot
        \overline{u_i^{(\sigma)}}
        =\zeta^{v_i^{(\sigma)}}\cdot
        \overline{\eta u_i^{(\sigma)}}.$$
Hence $\eta u_i^{(\sigma)}$ is an eigenvector of $\overline{\varsigma^s}$ associated to the eigenvalue $v_i^{(\sigma)}$. The proof is therefore complete.
\end{proof}


In the course of the proof, we have shown the following:
\begin{proposition}
The coarse moduli morphism $c_\sW\colon \sW\to W$ is an isomorphism above $W\setminus \widetilde{E}$, and for every $\mathscr{S}$ the base-change $\sW\times_W Q^\circ_\mathscr{S}\to Q^\circ_\mathscr{S}$ is a gerbe with constant inertia $\mu_{\gcd(\mathscr{S})}$.
\hfill \qedsymbol
\end{proposition}

We conclude with the following remarks.
\begin{remark}[Weighted blow-up of linear $\bZ/p$-actions]\label{rmk:weighted_blow-up_Z/p_action}
We have seen that well-chosen weighted blow-ups of $\bA^d$ resolve linear $\alpha_p$-actions (in the sense of \autoref{cor:pullback_foliation_is_regular}). One might ask if a similar phenomenon happens for linear $\bZ/p$-actions, especially in view of \autoref{prop:degen_of_Z/p_to_alpha_p}. If we select the weights in the same way, then the $\bZ/p$-action lifts to the \emph{schematic} weighted blow-up $Y\to \bA^d$. We prefer to work with equivariant charts, since $Y$ is singular. The action does not lift, however, to the charts $C_i^{(\sigma)}$. This can be arranged, at least in a neighbourhood of the exceptional divisor, by introducing $i$-th roots of $1+u_i^{(\sigma)}u_{i+1}^{(\sigma)}$. In any case one checks that the ramification locus of the $\bZ/p$-action is the exceptional divisor. 

We hoped that the ``$\bZ/p$-$\mu_p$ switch" phenomenon described by Totaro in \cite{Totaro_Terminal_3folds_not_CM} would apply on the locus where the ramification ideal is not principal. This is unfortunately not quite the case: on some charts we are lead to consider the non-linear action
        $$u_1\mapsto u_1+u_2u_d, \quad 
        u_2\mapsto u_2+u_3u_d, \quad \dots \quad 
        u_{d-1}\mapsto u_{d-1}+u_d^2, \quad u_d\mapsto u_d,$$
which does not lend itself to Totaro's method.

Let us also notice that in characteristic $0$, the weighted blow-up that we study was used by Ito and Reid in \cite{Ito_Reid_McKay_correspondence_SL3} to the prove a version of the McKay correspondence. We do not know if the characteristic $0$ and the characteristic $p$ pictures can be related in a precise way.
\end{remark}

\section{Stringy motivic invariants of linear $\alpha_p$-quotients}
In this section we compute the stringy motivic invariant of the quotient $\bA/(\alpha_p,\mathbf{d})$.

\subsection{Batyrev's formula}
The following formula, which is a special case of a result of Batyrev, will be our main tool in conjunction with the construction of a partial resolution in the previous section.

\begin{proposition}\label{prop:Batyrev_formula}
Let $(\Spec(\sO),\sum a_iB_i)$ be a cyclic quotient singularity of type $\frac{1}{n}(w_1,\dots,w_t)$ (cf.\ \autoref{def:cyclic_qt}). Assume that $a_i<1$ for every $i$. Let
        $$\varphi\colon (\bR_{\geq 0})^t\to \bR_{\geq 0}
        \quad \text{be the linear function defined by}\quad 
        \varphi(e_i)=1-a_i \ \forall i$$
where $e_i$ is the $i$-th standard basis vector. Let $N$ be the sub-group of $\bR^t$ generated by the vector $(w_1/n,\dots,w_t/n)^\intercal$ together with $e_1,\dots,e_d$. If $\mathbf{0}$ is the closed point of $\Spec(\sO)$ then
        $$\mst\left(\Spec(\sO),\sum_i a_iB_i\right)_\mathbf{0}=
        (\bL-1)^t\cdot \prod_{i=1}^t\frac{1}{1-\bL^{a_i-1}}
        \cdot \sum_{\mathfrak{p}\in N \cap (0;1]^t}
        \bL^{-\varphi(\mathfrak{p})}.
        $$
\end{proposition}
\begin{proof}
We first consider an affine $\mathbb{Q}$-factorial toric pair $(R,\sum a_i B_i')$ that has the same singularity type at a unique torus-invariant closed point denoted again by $\mathbf{0}$. Then, we claim that $\mst\left(R,\sum_i a_iB_i'\right)_\mathbf{0}$ is expressed by the same formula as in the statement; this will be sufficient to conclude, as the $\mst$ at a closed point is a formal-local invariant. This is a slight variation of  \cite[Theorem 4.5]{Batyrev_NA_integrals_and_stringy_Euler_numbers} and the proof is basically the same. We explain the outline of the proof. To show the claim, we first note that $\mst\left(R,\sum_i a_iB_i'\right)_\mathbf{0}$ is expressed as the motivic volume of the space of arcs that maps the closed point to $\mathbf{0}$. Since $R$ is a toric variety, the space of arcs of $R$ that maps the generic point into the open torus decomposes into countably many pieces indexed by $N\cap \sigma$. The arcs that does not map the generic point into the open torus has motivic measure zero and we can safely ignore them, when computing motivic volumes. Among such arcs, those that sends the closed point to $\mathbf{0}$ corresponds to the pieces indexed by $N\cap \sigma^\circ$. We can explicitly compute the contribution of each piece to the motivic volume, by applying the change of variables formula to a toric resolution and by a direct calculation on a simple normal crossing pair. The contribution is $(\bL-1)^t \bL^{-\varphi(\mathfrak{p})}$. Summing up them over $\mathfrak{p}\in N\cap \sigma^\circ$ we get the claim. 
\end{proof}

\subsection{Computation of the stringy motivic invariant}
We apply the formula just derived to the cyclic quotient singularities described in \autoref{prop:sing_partial_resol}, whose notations we keep. We introduce the following auxiliary function: 
For $y \in (0,1]$, we define 
$$
\theta(y) := 1 -\lfloor y \rfloor + \lfloor py \rfloor -\sum_{(\pi,j)}\lfloor jy\rfloor =
\begin{cases}
    1 + \lfloor py \rfloor -\sum_{(\pi,j)}\lfloor jy\rfloor & (y<1) \\
    p-\bD_{\mathbf{d}} & (y=1)
\end{cases}
$$
where $(\pi,j)$ runs through the elements of the set $\mathbf{I}=\{(\pi,j)\mid \pi=1,\dots,l \text{ and }j=0,\dots,d_\pi\}$. Recall that $\mathscr{S}$ denotes a non-empty subset of $\mathbf{I}^*=\{(\pi,j)\mid \pi=1,\dots,l \text{ and }j=1,\dots,d_\pi\}$.

\begin{proposition}\label{prop:local_Mst}
Assume that $\bD_\mathbf{d}> p-1$. Let $w\in Q^\circ_\mathscr{S}$ be a closed point, and assume that $g_\mathscr{S}=\gcd(\mathscr{S})>1$. Then:
        $$\mst\left(W,-(\bD_\mathbf{d}-p)\widetilde{E}\right)_w
        =(\bL-1)^d 
        \left( \frac{1}{1-\bL^{-1}} \right)^{d-1}
        \frac{\bL^{-d}}{1-\bL^{p-1-\bD_\mathbf{d}}}
        \left(
        \sum_{y\in \frac{1}{g_\mathscr{S}}\bZ\cap (0;1]}
        \bL^{\theta(y)}
        \right)
        $$
where $\theta$ is defined as above.
\end{proposition}
\begin{proof}
Write $\gamma=\bD_\mathbf{d}-p$ and $g=\gcd(\mathscr{S})$. Fix a pair $(\sigma,i)$ that belongs to $\mathscr{S}$. By \autoref{prop:sing_partial_resol} the point $w\in (W,-\gamma E)$ is of type
        $$\frac{1}{g}\left(v_j^{(\pi)}
        \mid (\pi,j)\in \mathbf{I}
        \right)$$
where 
        $$v_j^{(\pi)}=
            \begin{cases}
            p & \text{if }(\pi,j)=(\sigma,i-1),\\
            1 & \text{if }(\pi,j)=(\sigma,i),\\
            -j & \text{otherwise},
            \end{cases}$$
and the divisor $\widetilde{E}$ corresponds to the couple of indices $(\sigma,i)$. In the notation of \autoref{prop:Batyrev_formula}, the lattice $N$ is generated by the standard basis vectors together with the vector with entries $v_j^{(\pi)}/g$ (having fixed some order on the pairs $(\pi,j)\in \mathbf{I}$ for the rest of the proof). 

To describe $N$, we introduce the following notation: for $x\in \bQ$ we let
        $$\{x\}' =
        \begin{cases}
            \{x\}=x-\lfloor x\rfloor & \text{if }x\notin \bZ, \\
            1 & \text{if }x\in \bZ.
        \end{cases}$$
That being said, we observe that 
        $$N\cap (0;1]^d= \left\{ \mathbf{v}_n=
        \left(
        \left\{n \ \frac{v_j^{(\pi)}}{g}\right\}'
        \right)_{(\pi,j)\in \mathbf{I}}
        \mid n=0,\dots,g-1
        \right\};$$
indeed, all these elements are different since $v_i^{(\sigma)}=1$. Let $\varphi$ be the linear function associated to the cyclic quotient pair $w\in (W,-\gamma \widetilde{E})$ as in \autoref{prop:Batyrev_formula}. Then 
        $$\varphi(\mathbf{v}_0)=\varphi(\mathbf{1})=
        \gamma+\sum_{(\pi,j)}1 
        =\bD_\mathbf{d}-p+d$$ 
and for $g-1>n>0$:
        \begin{eqnarray*} 
        \varphi(\mathbf{v}_n)&=&
        \left\{n \ \frac{v_i^{(\sigma)}}{g}\right\}'(1+\gamma)+
        \sum_{(\pi,j)\neq (\sigma,i)}\left\{n \ \frac{v_j^{(\pi)}}{g}\right\}' \\
        &=& 
        \left\{ \frac{n}{g}\right\}'(1+\gamma)
        +\sum_{(\pi,j)\neq (\sigma,i)} \left\{\frac{-nj}{g}\right\}'
        -\left\{\frac{-n(i-1)}{g}\right\}'
        +\left\{\frac{np}{g}\right\}' \\
        &=&
        \sum_{(\pi,j)}\left\{\frac{-nj}{g}\right\}'
        -\left\{\frac{-ni}{g}\right\}'+\left\{\frac{n}{g}\right\}'(1+\gamma) -\left\{\frac{-n(i-1)}{g}\right\}'
        +\left\{\frac{np}{g}\right\}' \\
        &=&
        \sum_{(\pi,j)}\left\{\frac{-nj}{g}\right\}'
        -1+\left\{\frac{n}{g}\right\}'(1+\gamma) -\left\{\frac{n}{g}\right\}'
        +\left\{\frac{np}{g}\right\}'
        \end{eqnarray*}
where we have used, for the last equality, the fact that $g$ divides $i$.
Let us write $y=n/g$. Note that $\{y\}'=y$ and that $\{py\}'=\{py\}$ (as $n<g$ and $p$ is coprime with $g$). We can re-organize the above expression to get
        $$\varphi(\mathbf{v}_n)=
        \left[
            \sum_{(\pi,j)}\left\{-jy\right\}'+\gamma y
        \right]+
        \{py\}'-1.$$
Observe that by plugging $y=1$ on the right-hand side, we recover $\varphi(\mathbf{v}_0)$. So if we introduce the function
        $$\vartheta(z)=\left[
            \sum_{(\pi,j)\in \mathbf{I}}\left\{-jz\right\}'+\gamma z
        \right]+
        \{pz\}'-1, \quad \text{for }z\in (0;1],$$
by \autoref{prop:Batyrev_formula} we obtain that
    \begin{equation}\label{eqn:mst_intermeditate_expression}
    \mst\left(W,-(\bD_\mathbf{d}-p)\widetilde{E}\right)_w
        =(\bL-1)^d 
        \left( \frac{1}{1-\bL^{-1}} \right)^{d-1}
        \frac{1}{1-\bL^{p-1-\bD_\mathbf{d}}}
        \sum_{y\in \frac{1}{g}\bZ\cap (0;1]}
        \bL^{-\vartheta(y)}.
    \end{equation}
It remains to compare the functions $\theta$ and $\vartheta$. We can simplify the expression in brackets in the definition of $\vartheta$: using that
        $$-\lfloor x\rfloor =\lceil -x\rceil 
        \quad \text{and} \quad 
        \lceil x\rceil =\lfloor x\rfloor +\mathds{1}_{x\notin\bZ}
        \quad \text{for }x\in \bQ,$$
we have for $z\in (0;1]$:
        \begin{eqnarray*}
            \gamma z+\sum_{(\pi,j)}\{-jz\}' &=&
            \gamma z +\sum_\pi \sum_{j=0}^{d_\pi}\left(\{-jz\}+\mathds{1}_{jz\in \bZ}\right) \\
            &=& \gamma z+\sum_\pi\sum_j \left( 
            -jz-\lfloor -jz\rfloor +\mathds{1}_{jz\in \bZ}
            \right) \\
            &=& \gamma z -\sum_\pi \frac{d_\pi(d_\pi+1)}{2}z
            +\sum_\pi\sum_j \left( 
            \lceil jz\rceil +\mathds{1}_{jz\in \bZ}
            \right) \\
            &=& (\gamma -\bD_\mathbf{d})z+
            \sum_\pi\sum_j (\lfloor jz\rfloor +1) \\
            &=&
            -pz + d +\sum_{(\pi,j)}\lfloor jz\rfloor.
        \end{eqnarray*}
Therefore we find
        \begin{eqnarray*} 
        \vartheta(z)&=&
        \sum_{(\pi,j)}\lfloor jz\rfloor+\{pz\}'-pz -1 + d \\
        &=&
        \begin{cases}
        \sum_{(\pi,j)}\lfloor jz\rfloor-\lfloor pz\rfloor -1 + d &
        \text{if }z<1, \\
        \bD_\mathbf{d}-p+d & \text{if }z=1.
        \end{cases}
        \end{eqnarray*}
Hence $-\vartheta(z)=-d+\theta(z)$, and replacing into \autoref{eqn:mst_intermeditate_expression} yields the desired formula.
\end{proof}

\begin{lemma}\label{lemma:classes_strata}
In $\Kzerouh$, we have for a non-empty $\mathscr{S}$,
    $$[Q^\circ_\mathscr{S}]=
    \bL^l(\bL-1)^{|\mathscr{S}|-1}.$$
\end{lemma}
\begin{proof}
The first equality is obvious as $Q_\emptyset^\circ$ is empty. 
We show the second one. 
We have the decomposition of $E = \bP(\mathbf{d})\times \bA^l = \bigsqcup_{\mathscr{S}}P^\circ_{[\mathscr{S}]}$. Removing the factor $\bA^l$, we have the corresponding decompsition 
$\bP(\mathbf{d}) = \bigsqcup_{\mathscr{S}}R^\circ_{[\mathscr{S}]}$ with $P^\circ_{[\mathscr{S}]}= R^\circ_{[\mathscr{S}]}\times \bA^l$. 
This is nothing but the decomposition of the toric variety $\bP(\mathbf{d}) $ into the torus orbits (see also \autoref{lemma:equiv_description_stratification}). From an explicit description \cite[Examples 3.1.17 and 5.1.14]{Cox-Little-Schenck}, we easily see that the fan defining a weighted projective space is transformed to the fan defining a projective space of the same dimension by an $\bR$-linear automorphism. 
This shows that the orbit decomposition of $\bP(\mathbf{b})$ is similar to the one of a projective space. In particular, we see that $R^\circ_{[\mathscr{S}]} \cong (\bG_m)^{|\mathscr{S}|-1}$.  
Since the morphism $P^\circ_{[\mathscr{S}]}\to Q^\circ_{\mathscr{S}} $ is a universal homeomorphism, we have
$$
[Q^\circ_{\mathscr{S}}]=\left[P^\circ_{[\mathscr{S}]}\right]=[\bA^l\times (\bG_m)^{|\mathscr{S}|-1}] = \bL^l(\bL-1)^{|\mathscr{S}-1|}
$$
as claimed.
\end{proof}

For a positive integer $m$, we define  $F_m:=\left(\bigcup_{j=1}^{m} \frac{1}{j}\bZ \right)\cap (0,1]$. Note that the irreducible fractions in this set together with $0/1$, arranged in ascending order,  is called the Farey sequence of order $m$. 

\begin{theorem}\label{thm:mst_alpha_p_qt}
Assume that $\bD_\mathbf{d}>p-1$. 
Let $F:=F_{\max \mathbf{d}}$. For a positive integer $r$, let $\Lambda_r:=\{(\pi,i)\in \mathbf{I}^*\mid \text{$r$ divides $i$}\}$ and let $N_r$ be its cardinality. Then
\begin{align*}
\mst\big(\bA/(\alpha_p,\mathbf{d})\big)
=  \bL^d-\bL^l+  \frac{\bL^{l-1}}{1-\bL^{-1-\bD_\mathbf{d}+p}}\sum_{s/r\in F} (\bL^{N_r}-1) \bL^{\theta(s/r)}
\end{align*}
\end{theorem} 
\begin{proof}
Let $\gamma=\bD_\mathbf{d}-p$ and $X=\bA/(\alpha_p,\mathbf{d})$.
For a locally closed subset $C\subset X$, the stringy motivic invariant of a klt pair $(X,B)$ along $C$, $\mst(X,B)_C$, is the volume of $(J_\infty X)_C$ with respect to the Gorenstein motivic measure associated to $(X,B)$, where $(J_\infty X)_C$ is the space of arcs that send the closed point of the formal disk into $C$. 
Thus, if we have a decomposition $X=\bigsqcup_i C_i$ of $X$ into finitely many locally closed subsets $C_i$, then we have $\mst(X,B)=\sum_i \mst(X,B)_{C_i}$. Applying this to the specific decomposition obtained above of $W$, we get
        \begin{eqnarray*}
            \mst(X) = \mst(W, -\gamma \widetilde{E}) 
            =\ \bL^d-\bL^l+\sum_{\emptyset\neq\mathscr{S}\subseteq \mathbf{I}^*}\mst(W,-\gamma \widetilde{E})_{Q^\circ_\mathscr{S}} .
        \end{eqnarray*}
We now claim that for every non-empty $\mathscr{S}$,
        \begin{eqnarray*}
\mst(W,-\gamma \widetilde{E})_{Q^\circ_\mathscr{S}} =
 [Q^\circ_\mathscr{S}]\cdot \mst(W,-\gamma \widetilde{E})_{w_\mathscr{S}}
        \end{eqnarray*}
where $w_\mathscr{S}$ is any closed point of $Q^\circ_\mathscr{S}$.
To see this, we first note that if $(W,-\gamma \widetilde{E})$ were a toric pair and if $Q^\circ_{\mathscr{S}}$ were a torus orbit, then the equality would hold. Indeed, in this hypothetical situation, a torus isomorphic to  $Q^\circ_{\mathscr{S}}$ freely acts on $(J_\infty W)_{Q^\circ_{\mathscr{S}}}$. The computation of stringy motivic invariants is compatible with this action and we can easily deduce the equality. To show the equality in our actual situation, we use the fact that $(W,\widetilde{E})$ (or $(W,W\setminus \widetilde{E})$, depending on the convention) is a strict toroidal embedding. The pair is a toroidal embedding as we have already seen. From a footnote at page 195 of \cite{KKMSD_Toroidal_embeddings}
(cf.\ \cite[Corollary 4.8.3]{Wlodarczyk_functorial_resolution}),
we see that the pair is a strict toroidal embedding. Thus, for each point $w \in Q_{\mathscr{S}}^\circ$, there exists a Zariski neighborhood $w \in U\subset W$, a toric pair $(U',E')$ and an etale morphism $U\to U'$ that is compatible with  the local toric structures and the boundary divisors. Using this, we see that 
        \begin{eqnarray*}
\mst(W,-\gamma \widetilde{E})_{Q^\circ_\mathscr{S}\cap U} =
 [Q^\circ_\mathscr{S}\cap U]\cdot \mst(W,-\gamma \widetilde{E})_{w}. 
=
 [Q^\circ_\mathscr{S}\cap U]\cdot \mst(W,-\gamma \widetilde{E})_{w_{\mathscr{S}}}. 
        \end{eqnarray*}
We can cover $Q^\circ_{\mathscr{S}}$ by finitely many such open subsets $U_i$. From the inclusion-exclusion principle, we see that
\begin{align*}
  &  \mst(W,-\gamma \widetilde{E})_{Q^\circ_\mathscr{S}} \\
  &  = \sum _i \mst(W,-\gamma \widetilde{E})_{Q^\circ_\mathscr{S}\cap U_i} -\sum_{i<j}\mst(W,-\gamma \widetilde{E})_{Q^\circ_\mathscr{S}\cap U_i\cap U_j}+\sum_{i<j<k}\cdots \\
  &  = \mst(W,-\gamma \widetilde{E})_{w_{\mathscr{S}}}
  \left(\sum _i [Q^\circ_\mathscr{S}\cap U_i] -\sum_{i<j}[Q^\circ_\mathscr{S}\cap U_i\cap U_j]+\sum_{i<j<k}\cdots\right)\\
    &  = \mst(W,-\gamma \widetilde{E})_{w_{\mathscr{S}}}
  \cdot [Q^\circ_\mathscr{S}].
\end{align*}

Using \autoref{lemma:classes_strata} we obtain:
\begin{eqnarray}\label{eqn:mst_without_Farey}
\mst(X)&=&
\bL^d-\bL^l+ \sum_{\mathscr{S}\neq \emptyset}\bL^l(\bL-1)^{|\mathscr{S}|-1} \frac{\bL^{-1}(\bL-1)}{1-\bL^{-1-\gamma}}
        \left(\sum_{\frac{s}{r}\in \frac{1}{g_\mathscr{S}}\bZ\cap (0;1]}
        \bL^{\theta(s/r)}
        \right)  \nonumber \\ 
&=&  
\bL^d-\bL^l+ \sum_{\mathscr{S}\neq \emptyset} \frac{\bL^{l-1}(\bL-1)^{|\mathscr{S}|}}{1-\bL^{-1-\gamma}}
        \left(\sum_{\frac{s}{r}\in \frac{1}{g_\mathscr{S}}\bZ\cap (0;1]}
        \bL^{\theta(s/r)}
        \right) \nonumber \\
&=&  
\bL^d-\bL^l+  \frac{\bL^{l-1}}{1-\bL^{-1-\gamma}}\sum_{\mathscr{S}\neq \emptyset}
        (\bL-1)^{|\mathscr{S}|}\left(\sum_{\frac{s}{r}\in \frac{1}{g_\mathscr{S}}\bZ\cap (0;1]}
        \bL^{\theta(s/r)}
        \right).
\end{eqnarray}
An irreducible fraction $s/r$ appears when $\mathscr{S}$ is a subset of $\Lambda_r$. There are $\binom{N_r}{i}$ subsets of cardinality $i$. Thus, the above sum over $\mathscr{S}$ is rewritten as: 
\begin{align*}
&\sum_{\mathscr{S}\neq \emptyset}
        (\bL-1)^{|\mathscr{S}|}\left(\sum_{s/r\in (1/g_\mathscr{S})\bZ}
        \bL^{\theta(s/r)}
        \right) \\
&= \sum_{s/r} \sum_{i=1}^{N_r} \binom{N_r}{i}(\bL-1)^i \bL^{\theta(s/r)}\\
&= \sum_{s/r} (((\bL-1)+1)^{N_r}-1) \bL^{\theta(s/r)}\\
&= \sum_{s/r} (\bL^{N_r}-1) \bL^{\theta(s/r)},
\end{align*}
which concludes.
\end{proof} 

\section{Comparison with linear $\bZ/p$-quotients}\label{section:comparison}

In \cite{Tonini_Yasuda_Motivic_McKay_cor_for_alpha_p}, it was conjectured that the two quotient varieties $\bA/(\alpha_p,\mathbf{d})$ and 
$\bA/(\bZ/p,\mathbf{d})$ associated to the same $\mathbf{d}$ have the same stringy motivic invariant in the version of the complete Grothendieck ring used in \cite{Yasuda_p-cyclic_McKay_correspondence}:
$$
\mst(\bA/(\alpha_p,\mathbf{d}))=\mst(\bA/(\bZ/p,\mathbf{d})). 
$$
An approach to this conjecture using motivic integration was proposed in \cite{Tonini_Yasuda_Motivic_McKay_cor_for_alpha_p} and extended in \cite{YamamotoThesis2024}, but it remains incomplete. In this section, we give partial results on this conjecture via a different approach by comparing the explicit formula for $\mst(\bA/(\alpha_p,\mathbf{d}))$ given in \autoref{thm:mst_alpha_p_qt} with the one for $\mst(\bA/(\bZ/p,\mathbf{d}))$  given in \cite{Yasuda_p-cyclic_McKay_correspondence}. This approach will be completed in \autoref{section:appendix}.

In \cite{Tonini_Yasuda_Motivic_McKay_cor_for_alpha_p}, the equality of stringy motivic invariants along the images $\mathbf{0}$ of the origin is also conjectured:
let us observe that it is a consequence of (in fact, equivalent to) the equality of the full stringy motivic invariants.

\begin{proposition}
We have
    $$\mst(\bA/(\alpha_p,\mathbf{d}))=
\mst(\bA/(\bZ/p,\mathbf{d}))
\quad \Longleftrightarrow 
\quad 
\mst(\bA/(\alpha_p,\mathbf{d}))_\mathbf{0}
=\mst(\bA/(\bZ/p,\mathbf{d}))_\mathbf{0}.$$
\end{proposition}
\begin{proof}
More precisely, we claim that for $G$ either $\bZ/p$ or $\alpha_p$, we have
    $$\mst(\bA/(G,\mathbf{d}))=
        \left[\bA^d\setminus V\left(
        x_i^{(\sigma)}\mid (\sigma,i)\in \mathbf{I}^*
        \right)\right]
        +[\bA^l]\cdot \mst(\bA/(G,\mathbf{d}))_\mathbf{0}.$$
The equivalence of the statement follows immediately from these expressions. Let us give a geometric argument that applies to both cases at the same time. 


Let $\mathbf{R}$ be the image of $\mathfrak{R}=V(x_i^{(\sigma)}\mid (\sigma,i)\in \mathbf{I}^*)$ in $X=\bA/(G,\mathbf{d})$. Since the $G$-action is free outside $\mathfrak{R}$, we have
        $$\mst(X)=
        \left[\bA^d\setminus \mathfrak{R}\right]
        +\mst(X)_\mathbf{R}.$$
We let the additive group $\bG_a^l$ act by translation on the coordinates $x_{0}^{(\sigma)}$ for $\sigma=1,\dots,l$. This gives a free action on $\bA^d$ that commutes with the $G$-action, and leaves $\mathfrak{R}$ stable. So the $\bG_a^l$-action descends to the quotient $X$ and leaves $\mathbf{R}$ stable. Hence the action of $\bG_a^l$ on the arcs of $X$, by way of translating the base-points, preserves the space $J_\infty(\bA^d/G)_\mathbf{R}$ of arcs meeting $\mathbf{R}$.

\begin{claim}
The $\bG_a^l$-action on $\mathbf{R}$ is transitive and set-theoretically free.
\end{claim}
\begin{proof}\renewcommand{\qedsymbol}{$\lozenge$}
The morphism $\mathfrak{R}\to \mathbf{R}$ is bijective and $\bG_a^l$-equivariant. As the action of $\bG_a^l$ on $\mathfrak{R}$ is transitive and free by design, the $\bG_a^l$-action on $\mathbf{R}$ is transitive and set-theoretically free (i.e.\ its stabilizers are infinitesimal group schemes).
\end{proof}

It follows from the claim that the action map
        $$\bG_a^l \times J_\infty(X)_\mathbf{0}
        \longrightarrow J_\infty(X)_\mathbf{R}$$
is a universal homeomorphism. Let $I$ be the ideal sheaf on $X$ defined as in \autoref{section:mst}. The function $\ord I$ is invariant on $\bG_a^l$-orbits. So in $\Muh$ we obtain:
    \begin{eqnarray*}
        \mst(X)_\mathbf{R} &=& 
        \int_{J_\infty(X)_\mathbf{R}}\mathbb{L}^{\ord I}\ d\mu_X \\
        &=& \int_{\bG_a^l \times J_\infty(X)_\mathbf{0}}
        \mathbb{L}^{\ord I} \ d\mu_X \\
        &=& [\bG_a^l] \cdot \int_{J_\infty(X)_\mathbf{0}}
        \mathbb{L}^{\ord I} \ d\mu_X \\
        &=& [\bG_a^l] \cdot \mst(X)_\mathbf{0}
    \end{eqnarray*}
and the equality claimed at the beginning follows.
\end{proof}

Hence, for the rest of the section, we concentrate on the comparison of full stringy motivic invariants.

There is a slight difference in notations between the present article and \cite{Yasuda_p-cyclic_McKay_correspondence}, which we address now. Take $\mathbf{d}=(d_1,\dots,d_l)$ subject to the conditions of \autoref{section:notations}, and write
        $$\mathbf{d}^+=(d_1+1,\dots,d_l+1)=(d_1^+,\dots,d_l^+).$$
Observe that $d=\sum_\lambda (1+d_\lambda)=\sum_\lambda d_\lambda^+$. The linear action of $\bZ/p$ on $\bA^d$ corresponding to $\mathbf{d}$ in our terminology, is the $\bZ/p$-linear action on $\bA^d$ associated to $\mathbf{d}^+$ in \cite{Yasuda_p-cyclic_McKay_correspondence} (up to permutation of coordinates). In \cite{Yasuda_p-cyclic_McKay_correspondence} the author introduces the quantity
        $$D_{\mathbf{d}^+}=\sum_\lambda \frac{d_\lambda^+(d_\lambda^+-1)}{2}
        =\sum_\lambda \frac{d_\lambda(d_\lambda+1)}{2}
        =\bD_\mathbf{d}$$
and the function
$$
\mathrm{sht}_{\mathbf{d}^+}(j) = \mathrm{sht}(j) :=
\sum_{\lambda=1}^l \sum_{i=1}^{d_\lambda^+-1} \left\lfloor \frac{ij}{p} \right\rfloor
=\sum_{\lambda=1}^l \sum_{i=1}^{d_\lambda} \left\lfloor \frac{ij}{p} \right\rfloor.
$$
If $X$ is normal $\bQ$-Gorenstein but not klt, by convention $\mst(X)$ is undefined. Moreover, if $\mathbb{D}_\mathbf{d}=1$ then both $\bA/(\bZ/p,\mathbf{d})$ and $\bA/(\alpha_p,\mathbf{d})$ are regular, and in this case their stringy motivic invariant in $\Muh$ is equal to $[\bA^d]$.
So in view of 
\autoref{thm:MMP_sing_alpha_p_qt} and \cite[Corollary 6.19]{Yasuda_p-cyclic_McKay_correspondence}, taking in account that $\bD_\mathbf{d}=D_{\mathbf{d}^+}$, we may assume for the rest of the section that $\bD_\mathbf{d}\geq p$. 

Then, we have
$$
\mst(\bA/(\bZ/p,\mathbf{d})) =
\bL^d +\frac{\bL^{l-1}(\bL-1)\sum_{j=1}^{p-1}\bL^{j-\mathrm{sht}(j)}}{1-\bL^{p-1-D_{\mathbf{d}}}}.
$$
Thus, both of $
\mst(\bA/(\alpha_p,\mathbf{d}))$ and $\mst(\bA/(\bZ/p,\mathbf{d}))$  are rational functions in $\bL$, and the above conjecture is equivalent to saying that the two rational functions are identical. We show that the conjecture is reduced to the following combinatorial statement:

\begin{conjecture}\label{conj:equality_of_multisets}
Assume that $\gamma=\bD_\mathbf{d}-p\geq 0$. Define $F=F_{\max\mathbf{d}}$ and $N_r$ as in \autoref{thm:mst_alpha_p_qt}. Then we have the following equality of multi-sets:
\begin{align*}
&\left \{ \left\{j -\mathrm{sht}(j)\mid 1\le j\le p-1 \right\}\right\} \uplus
\left \{ \left\{0,-1,\dots,-\gamma \right\}\right\}\\
&=
   \left \{ \left\{\theta\left(\frac{s}{r}\right)+i \mid \frac{s}{r} \in F,\, 0\le i \le N_r-1\right\}\right\}
\end{align*}
\end{conjecture}

\begin{proposition}\label{prop:equivalence_of_conj}
The above conjecture is equivalent to $
\mst(\bA/(\alpha_p,\mathbf{d}))=\mst(\bA/(\bZ/p,\mathbf{d}))
$.
\end{proposition}

\begin{proof}
Let $L$ and $R$ denote the multi-sets on the left hand side and the right hand side of the conjectured equality, respectively.
The proposition is proved by transforming 
the desired equality as follows:
\begin{align*}
&\mst(\bA/(\alpha_p,\mathbf{d}))=\mst(\bA/(\bZ/p,\mathbf{d})) \\
&\Leftrightarrow \bL^d +\frac{\bL^{l-1}(\bL-1)\sum_{j=1}^{p-1}\bL^{j-\mathrm{sht}(j)}}{1-\bL^{-1-\gamma}}=\bL^d-\bL^l+  \frac{\bL^{l-1}}{1-\bL^{-1-\gamma}}\sum_{s/r\in F} (\bL^{N_r}-1) \bL^{\theta(s/r)} \\
&\Leftrightarrow (\bL-1)\sum_{j=1}^{p-1}\bL^{j-\mathrm{sht}(j)}=
-\bL(1-\bL^{-1-\gamma})+
        \sum_{s/r\in F} (\bL^{N_r}-1) \bL^{\theta(s/r)} \\
&\Leftrightarrow \sum_{j=1}^{p-1}\bL^{j-\mathrm{sht}(j)}
+ \frac{\bL(1-\bL^{-1-\gamma})}{\bL-1}
=\sum_{s/r\in F} \frac{\bL^{N_r}-1}{\bL-1}\bL^{\theta(s/r)} \\
&\Leftrightarrow \sum_{j=1}^{p-1}\bL^{j-\mathrm{sht}(j)}
+ (1+\bL^{-1} +\cdots \bL^{-\gamma})
=\sum_{s/r\in F} (\bL^{N_r-1}+\bL^{N_r-2}+\cdots +1) \bL^{\theta(s/r)}\\
&\Leftrightarrow \sum_{i\in L}\bL^i
=\sum_{i\in R}\bL^i.
\end{align*}
This proves the claimed equivalence.
\end{proof}



\begin{example}\label{ex:multi-sets}
Consider the case where $p=7$ and $\mathbf{d}^+=(6)$. 
Then, a direct computation shows that $\gamma = 8$
and
$$
i-\mathrm{sht}(i)= 2-i.
$$
Thus, the multi-set on the left hand side is
\begin{align*}
    &\{\{1,0,-1,-2,-3,-4\}\}\uplus \{\{0,-1,-2,\dots,-8\}\}\\
&= \{\{1,0,0,-1,-1,-2,-2,-3,-3,-4,-4,-5,-6,-7,-8\}\}.
\end{align*}
Computation of  the multiset on the right hand side is summarized in the following table:
\begin{center}
\begin{tabular}{|c|c|c|c| } 
 \hline
$s/r$ & $\theta(s/r)$ & $N_r$ & $\theta(s/r)+i$ $(0\le i\le N_r-1)$\\ 
 \hline
1/5& 1 & 1 & 1 \\ 
 \hline
1/4& 0 & 1 & 0 \\ 
 \hline
1/3& 0 & 1 & 0 \\ 
 \hline
2/5& $-1$ & 1 & $-1$ \\ 
 \hline
1/2& $-2$ & 2 & $-2$, $-1$ \\ 
 \hline
3/5& $-2$ & 1 & $-2$ \\ 
 \hline
2/3& $-3$ & 1 & $-3$ \\ 
 \hline
3/4& $-3$ & 1 & $-3$\\ 
 \hline
4/5& $-4$ & 1 & $-4$ \\ 
 \hline
1/1& $-8$ & 5 & $-8$, $-7$, $-6$, $-5$, $-4$,  \\ 
 \hline
\end{tabular}
\end{center}
In particular, we see that the above conjecture on multi-sets holds in this case. 
\end{example}

The \autoref{conj:equality_of_multisets}, and thus the unconditional equality of stringy motivic invariants, will be proven in \autoref{section:appendix}: the argument is due to Linus Rösler, who contacted the first two authors after the first version of this work was made available. Below are some supporting evidence of this conjecture. First, by a computer-assisted proof, we can show the following proposition. 

\begin{proposition}\label{prop:comparison_mst}
Assume that one of the following holds:
\begin{enumerate}
    \item $p\le 131$ and $|\mathbf{d}^+|=|d_1^++\cdots+d_l^+| \le 32$.
    \item $p\le 173$ and $\mathbf{d}^+=(d_1^+)$.
\end{enumerate}
Assume also $D_{\mathbf{d}^+}\ge p$.  Then, Conjecture \ref{conj:equality_of_multisets} holds and hence $\mst(\bA/(\alpha_p,\mathbf{d}) )=\mst(\bA/(\bZ/p,\mathbf{d}) ).$
\end{proposition}

\begin{proof}
Our proof is computer-assisted. We simply verified Conjecture \ref{conj:equality_of_multisets} in all the cases satisfying the assumption by using  the software Mathematica (Version 14.2)  \cite{Mathematica}. We will explain  the actual Mathematica code that we used for this computation.  We first create a function \verb|multi1| to compute the multi-set on the left hand side as follows:
\begin{Verbatim}[frame=single, breaklines=true, breakafter={,}]
In[1]:= sht[j_,dim_Integer,prime_] := Sum[Floor[i*j/prime],{i,1,dim-1}]
In[2]:= sht[j_,dPlus_List,prime_]:=Sum[sht[j,dPlus[[i]],prime], {i,1,Length[dPlus]}]
In[3]:= boldD[dPlus_]:=Sum[dPlus[[i]](dPlus[[i]]-1)/2,{i,1,Length[dPlus]}]
In[4]:= multi1[dPlus_,prime_]:=Join[Table[j-sht[j,dPlus,prime],{j,1,prime-1}],Range[-boldD[dPlus]+prime,0]]//Sort
\end{Verbatim}
Note that we compute a multi-set as a list of integers arranged in ascending order. Next we  create a function \verb|multi2| to compute the multi-set on the right hand side:
\begin{Verbatim}[frame=single, breaklines=true, breakafter={,}]
In[5]:= numDiv[r_,d_Integer]:=Floor[d/r]
In[6]:= numDiv[r_,d_List]:=Sum[numDiv[r,d[[i]]],{i,1,Length[d]}]
In[7]:= theta[y_,d_List,prime_]:=1-Floor[y]+Floor[prime*y]-Sum[Floor[j *y],{i,1,Length[d]},{j,1,d[[i]]}]
In[8]:= multi2[d_List,prime_Integer]:=Module[{x,farey,len,myList={}},
farey=FareySequence[Max[d]];len=Length[farey];
For[s=2,s<=len,s++,y=farey[[s]];
myList=Append[myList,Table[theta[y,d,prime]+i,{i,0,numDiv[Denominator[y],d]-1}]]];
myList//Flatten//Sort]
In[9]:= multi1[{6},7]
Out[9]= {-8,-7,-6,-5,-4,-4,-3,-3,-2,-2,-1,-1,0,0,1}
In[10]:= multi2[{5},7]
Out[10]= {-8,-7,-6,-5,-4,-4,-3,-3,-2,-2,-1,-1,0,0,1}
\end{Verbatim}
Here \verb|numDiv| is a function to compute $N_r$. The last two inputs and outputs compute multi-sets in the situation of Example \ref{ex:multi-sets}. We then check Conjecture \ref{conj:equality_of_multisets} for the cases where $p$ is at most $131$, which is the 32nd prime number, and $|\mathbf{d}^+| \le 32 $ as follows:
\begin{Verbatim}[frame=single, breaklines=true, breakafter={,}]
In[11]:= validDimPartitions[dim_,prime_]:=Select[IntegerPartitions[dim,All,Range[2,prime]], (boldD[#]>=prime)&]
In[12]:= checkConj[dPlus_List,prime_Integer]:=multi1[dPlus,prime]==multi2[dPlus-1,prime]
In[13]:= checkConj2[totalDim_Integer,prime_Integer]:=AllTrue[validDimPartitions[totalDim,prime],checkConj[#,prime]&]
In[14]:= checkConj3[dimMax_Integer,primeOrderMax_Integer]:=AllTrue[Flatten[Table[{totalDim,Prime[primeOrder]},{totalDim,1,dimMax},{primeOrder,1,primeOrderMax}],1],checkConj2[#[[1]],#[[2]]]&,1]
In[15]:= checkConj3[32,32]
Out[15]= True
\end{Verbatim}
The function \verb|validDimPartitions| produces all the possible partitions $d_1^+\ge \cdots \ge d_\lambda^+$ of a given positive integer into integers from $2$ to $p$ such that $D_{\mathbf{d}^+}\ge p$. Note that we omit 1, since it corresponds to the trivial indecomposable representation and does not contribute to the computation of stringy motivic invariants except multipying them with $\bL$. The last input took about 9 minutes to process on a MacBook Pro (2020) with an Intel Core i7 processor and 32 GB of memory.
Lastly we check the indecomposable cases for $p\le 173$, where $173$ is the 40th prime number, as follows:
\begin{Verbatim}[frame=single, breaklines=true, breakafter={,}]
In[16]:= checkConjIndecomp[primeOrderMax_Integer]:=AllTrue[Flatten[Table[{dim,Prime[primeOrder]},{primeOrder,1,primeOrderMax},{dim,3,Prime[primeOrder]}],1],(#[[1]](#[[1]]-1)/2<#[[2]]||checkConj[{#[[1]]},#[[2]]])&]
In[17]:= checkConjIndecomp[40]
Out[17]= True
\end{Verbatim}
In the above computation, we only check indecomposable representations of dimension between $3$ and $p$. This is because the indecomposable representations of dimensions 1 and 2 do not satisfy the condition  $\bD_{\mathbf{d}} \ge p$. 
The last input took about 8 minutes to process.
\end{proof}

The second supporting evidence of \autoref{conj:Tonini_Yasuda} is the coincidence of stringy Euler numbers. The stringy Euler number $\est(X)$ of a variety $X$ with only klt singularities is defined to be the limit $\lim_{u,v\to 1} \Est(X;u,v)$. When $X$ is either $\bA/(\alpha_p,\mathbf{d})$ or  $\bA/(\bZ/p,\mathbf{d})$ as above, then $\mst(X)$ is a rational function in $\bL$ which does not have a pole at $\bL=1$. In this case, the stringy Euler number is nothing but the value of this rational function at $\bL=1$.

\begin{proposition}\label{prop:comparison_Euler_numbers}
Assume that $\bD_\mathbf{d}\geq p$. Then we have 
$$\est(\bA/(\alpha_p,\mathbf{d}) )=\frac{\bD_{\mathbf{d}}}{\bD_\mathbf{d}-p+1}=\est(\bA/(\bZ/p,\mathbf{d}) ).$$
\end{proposition}

\begin{proof}
We first compute $\est(\bA/(\alpha_p,\mathbf{d}) )$. As usual, let $\gamma=\bD_\mathbf{d}-p$. From \autoref{eqn:mst_without_Farey} we obtain
\begin{align*}
  \est(\bA/(\alpha_p,\mathbf{d}) )  
  = \left( \bL^d-\bL^l+ \sum_{\emptyset\neq \mathscr{S}\subseteq \mathbf{I}^*} \frac{\bL^{l-1}(\bL-1)^{|\mathscr{S}|}}{1-\bL^{-1-\gamma}}
        \left(\sum_{\frac{s}{r}\in \frac{1}{g_\mathscr{S}}\bZ\cap (0;1]}
        \bL^{\theta(s/r)}
        \right)\right)|_{\bL=1}. 
\end{align*}
Here we note that the fraction with $|\mathscr{S}|>1$ reduces to zero when substituting $\bL$ with $1$.
Thus, we can continue as:
\begin{align*}
\est(\bA/(\alpha_p,\mathbf{d}) ) & =  \sum_{(\pi,i)\in \mathbf{I}^*} \frac{1}{\gamma+1}
        \left(\sum_{s/r\in (1/i)\bZ\cap(0,1]}
        1
        \right) \\
& = \frac{\sum_{(\pi,i)}i}{\gamma+1} \\
&= \frac{\bD_{\mathbf{d}}}{\gamma+1}. 
\end{align*}
Next we compute $\est(\bA/(\bZ/p,\mathbf{d}) )$ as follows (see also \cite[Corollary 6.12]{Yasuda_p-cyclic_McKay_correspondence}):
\begin{eqnarray*}
  \est(\bA/(\bZ/p,\mathbf{d}) )  
 & = &
 \left( \bL^d +\frac{\bL^{l-1}(\bL-1)\sum_{j=1}^{p-1}\bL^{j-\mathrm{sht}(j)}}{1-\bL^{-1-\gamma}}\right)|_{\bL=1}\\
 & =& 
 \left(1+  \frac{\sum_{j=1}^{p-1}\bL^{j-\mathrm{sht}(j)}}{\bL^{\gamma}+\bL^{\gamma-1}+\cdots+1}\right)|_{\bL=1}\\
 & = &
 1+  \frac{p-1}{\gamma+1}\\
 & = &
 \frac{\bD_{\mathbf{d}}}{\gamma+1}.
\end{eqnarray*}
This proves the statement.
\end{proof}

The proof of the equality 
$
\mst(\bA/(\alpha_p,\mathbf{d}))=\mst(\bA/(\bZ/p,\mathbf{d}))
$ 
is completed in the appendix below. This suggests that the family $\mathcal{X}=(\bA^d \times \bA^1)/\mathbf{G}\to \bA^1$ constructed in Section \ref{section:degeneration}, which degenerates the quotient varieties $\bA/(\alpha_p,\mathbf{d}))$ and $\bA/(\bZ/p,\mathbf{d})$ (at least up to universal homeomorphism), is ``equisingular'' in some sense. Note that weighted blow-up of $\bA^d\times \bA^1$ with the weights indicated in \autoref{section:weighted_blow-up} does not simplify the $\mathbf{G}$-action sufficiently to induce a simultaneous stacky resolution of $\sX\to \bA^1$, see \autoref{rmk:weighted_blow-up_Z/p_action}. 
With this in mind, it seems natural to  ask the following question, even though whether $\bA/(\bZ/p,\mathbf{d})$ has a resolution is, to the best of our knowledge, an open question:

\begin{question}\label{question:simultaneous_resolution}
Does the family $\mathcal{X}\to \bA^1$ admit a simultaneous resolution?
\end{question}

\appendix
\section{A proof of 
\autoref{conj:equality_of_multisets} --- 
by Linus Rösler}\label{section:appendix}

\subsection{Notations}
    We keep the main notations from the text: $\mathbf{d}=\{d_\lambda\}$ and $\mathbf{I}^*$ are as in \autoref{section:notations}, $\Lambda_r$ and $N_r=|\Lambda_r|$ are as in \autoref{thm:mst_alpha_p_qt}. We write $\gamma =\mathbb{D}_\mathbf{d}-p$.
    Recall that the functions $\theta$ and $\sht$ are defined by
        $$\theta(y)= 1-\lfloor y\rfloor+\lfloor py\rfloor-\sum_{(\lambda,i)\in\bfI^*}\lfloor iy\rfloor 
        \quad \text{for }y\in (0;1]$$
    and 
        $$\sht(j)=\sum_{(\lambda,i)\in\bfI^*}\left\lfloor\frac{ij}{p}\right\rfloor
        \quad \text{for }j\in \bZ_{>0}.$$
    Let $M=\max_\lambda d_\lambda$. Recall the definition of the Farey sequence of order $M$,
        $$F=\left(\bigcup_{1\leq j\leq M}\frac{1}{j}\bZ\right)\cap (0,1]$$
    which we partition as follows:
        $$F^{(j)}=F\cap\left[\frac{j}{p},\frac{j+1}{p}\right)=\{y_{j1}<\cdots<y_{j,k_j}\}
        \quad \text{for }
        j=0,\dots,p.$$   
    Finally, we denote by $\sL$ resp. $\sR$ the multisets
    \begin{align*}
        \sL&\coloneqq\{\{j-\sht(j)\mid 1\leq j<p\}\}\uplus\llbracket -\gamma,0\rrbracket\\
        \sR&\coloneqq\left\{\left\{\left.\theta\left(\frac{s}{r}\right)+i\ \right|\ \frac{s}{r}\in F,\ \gcd(s,r)=1,\ 0\leq i<N_r\right\}\right\}
    \end{align*}
   appearing in \autoref{conj:equality_of_multisets}, where for integers $m\leq n$ we write
   $\llbracket m,n\rrbracket =[m;n]\cap \bZ$.

    \subsection{Proof of 
\autoref{conj:equality_of_multisets}}
    Let $a$ be the minimal non--negative integer such that $F^{(a)}\neq\emptyset$. We start by ruling out a trivial edge case.
    \begin{lemma}\label{lem:trivial_case}
        If $a=p$ (or, equivalently, $M=1$), then \autoref{conj:equality_of_multisets} is true.
    \end{lemma}

    \begin{proof}
        In this case we must have $M=1$, which means that $\bfd$ consists of $1$'s and $0$'s (with at least $p$ entries equal to $1$). We then obtain that
        \begin{align*}
            \sht(j)=\bD_{\bfd}\left\lfloor \frac{j}{p}\right\rfloor
        \end{align*}
        which is equal to $0$ for all $1\leq j<p$. Hence we obtain that $\sL=\llbracket -\gamma,p-1\rrbracket$.
        On the other hand, as $M=1$ we obtain $F=\{1\}$, so that $\sR=\llbracket\theta(1),\theta(1)+N_1-1\rrbracket$.
        Observe that $\theta(1)=-\gamma$ and $N_1=\bD_{\bfd}$, so that
        \begin{align*}
            \theta(1)+N_1-1=-(\bD_{\bfd}-p)+\bD_{\bfd}-1=p-1,
        \end{align*}
        which concludes the proof.
    \end{proof}

    So from now on we assume $a<p$ and $M>1$. In particular, we also have $p>2$ (as otherwise $M=1$), which implies that $a< p/2$. Note also that $a\geq 1$, as $1/M>1/p$. Observe that $y\mapsto 1-y$ is a symmetry of $F\setminus\{1\}$, and as $M<p$, we have $F^{(j)}=F\cap(j/p,(j+1)/p)$ for all $1\leq j<p$. In particular, $y\mapsto 1-y$ induces a bijection between $F^{(j)}$ and $F^{(p-j-1)}$ for all $1\leq j<p$.

    We will need to denote precisely which $F^{(j)}$'s are non--empty and which $F^{(j)}$'s are empty. Therefore, we introduce the intertwined sequences
    \begin{align*}
        a=a_0\leq b_0<a_1\leq b_1<\cdots < a_n\leq b_n=p-a-1,
    \end{align*}
    such that $F^{(j)}$ is non--empty if and only if $j\in\llbracket a_m,b_m\rrbracket$, and such that $b_m+1\leq a_{m+1}-1$ for all $0\leq m<n$ (i.e., the intervals $\llbracket a_m,b_m\rrbracket$ are maximal). Because of the bijection between $F^{(j)}$ and $F^{(p-j-1)}$, it is straightforward to see that
    \begin{align*}
        b_{n-m}=p-a_m-1
    \end{align*}
    for all $0\leq m\leq n$. To summarize, $F^{(j)}$ is non-empty precisely when $a_m\leq j\leq b_m$ for some $0\leq m\leq n$ or $j=p$ (in which case $F^{(p)}=\{1\}$).

    We now describe the function $\sht$ through the lens of the sequences $\{a_m\}$ and $\{b_m\}$, and thereby we also obtain a description of $\sL$.

    \begin{lemma}\label{lem:sht}
        The function $\sht$ has the following properties.
        \begin{enumerate}
            \item\label{lem:sht_reflection} For all $1\leq j<p$ we have
            \begin{align*}
                \sht(p-j)=\bD_{\bfd}-N_1-\sht(j).
            \end{align*}
            \item\label{lem:sht_border} We have
            \begin{align*}
                \forall j\in\llbracket1,a\rrbracket:&\quad\sht(j)=0\\
                \forall j\in\llbracket p-a,p-1\rrbracket:&\quad\sht(j)=\bD_{\bfd}-N_1
            \end{align*}

            \item\label{lem:sht_constant} For all $0\leq m<n$, the function $\sht$ is constant on the interval $\llbracket b_m+1,a_{m+1}\rrbracket$.
            \item\label{lem:sht_non_increasing} For all $0\leq m\leq n$, the function
            \begin{align*}
                j\in \llbracket a_m,b_m+1\rrbracket\mapsto j-\sht(j)
            \end{align*}
            is non--increasing.
        \end{enumerate}
    \end{lemma}

    \begin{proof}
        From the definition of $\sht$ we immediately obtain
        \begin{align*}
            \sht(p-j)&=\sum_{(\lambda,i)\in\bfI^*}\left(i-\left\lceil\frac{ij}{p}\right\rceil\right)\\
            &=\bD_{\bfd}-\sum_{(\lambda,i)\in\bfI^*}\left\lceil\frac{ij}{p}\right\rceil\\
            &=\bD_{\bfd}-\sum_{(\lambda,i)\in\bfI^*}\left(\left\lfloor\frac{ij}{p}\right\rfloor+1\right)\\
            &=\bD_{\bfd}-N_1-\sht(j),
        \end{align*}
        where for the third equality we used that $i,j<p$. This proves point \autoref{lem:sht_reflection}.

        For point \autoref{lem:sht_border}, note that as $1/M\in F^{(a)}$, we obtain $a/p<1/M$. Thus, for $1\leq j\leq a$ and $(\lambda,i)\in \bfI^{*}$, it follows that $ij/p<i/M\leq 1$, from which we deduce that $\sht(j)=0$. By applying \autoref{lem:sht_reflection}, we conclude \autoref{lem:sht_border}.

        In order to prove \autoref{lem:sht_constant}, observe the following: for $1\leq j\leq p-1$, if $F^{(j)}$ is empty, then $\sht(j)=\sht(j+1)$. Indeed, consider the function
        $$t\colon (0,1] \to\bZ_{\geq 0}, \quad 
        y \mapsto \sum_{(\lambda,i)\in\bfI^*}\left\lfloor i y\right\rfloor,$$
        so that $\sht(j)=t(j/p)$ for all $j$. Notice that $t$ is right--continuous, and it is discontinuous precisely at the points in $F$. Also, between its discontinuities, $t$ is constant. As $p>M$, we have $F\cap (1/p)\bZ=\{1\}$, so $t$ is continuous at $j/p$ for $1\leq j\leq p-1$. Furthermore, if $F^{(j)}$ is empty, this precisely means that no element of $F$ lies between $j/p$ and $(j+1)/p$. Hence in this case $t$ is continuous on $[j/p,(j+1)/p]$, and thus constant. Therefore we obtain
        \begin{align*}
            \sht(j)=t(j/p)=t((j+1)/p)=\sht(j+1).
        \end{align*}
        By induction, using that $F^{(j)}$ is empty for $j\in\llbracket b_m+1, a_{m+1}-1\rrbracket$, point \autoref{lem:sht_constant} follows.
        
        Finally we come to point \autoref{lem:sht_non_increasing}. We start by showing that if $F^{(j)}$ is non--empty, then
        \begin{align*}
            j-\sht(j)\geq j+1-\sht(j+1),
        \end{align*}
        or equivalently that $\sht(j+1)-\sht(j)\geq 1$. Indeed, if $F^{(j)}$ is non--empty, there exist positive integers $i'<i\leq M$ such that
        \begin{align*}
            \frac{j}{p}< \frac{i'}{i}<\frac{j+1}{p}.
        \end{align*}
        In particular, we have
        \begin{align*}
            \left\lfloor\frac{ij}{p}\right\rfloor<i'\leq \left\lfloor\frac{i(j+1)}{p}\right\rfloor,
        \end{align*}
        from which we immediately deduce that $\sht(j+1)-\sht(j)\geq 1$. Then, point \autoref{lem:sht_non_increasing} follows by induction, noting that $F^{(j)}$ is non--empty for all $j\in\llbracket a_m,b_m\rrbracket$.
    \end{proof}

    With this precise description of $\sht$, we can describe $\sL$ more concretely.

    \begin{lemma}\label{lem:LHS}
        We have
        \begin{align*}
            \sL=\{\{j-\sht(j)\mid 0\leq m\leq n,\  a_m<j\leq b_m\}\}\uplus\llbracket -\gamma-a+N_1,-\gamma+N_1-1\rrbracket\uplus \llbracket -\gamma,a\rrbracket\\
            \uplus\biguplus_{0\leq m<n}\llbracket b_m+1-\sht(a_{m+1}),a_{m+1}-\sht(a_{m+1})\rrbracket.
        \end{align*}
    \end{lemma}

    \begin{proof}
        The interval $\llbracket -\gamma,a\rrbracket$ comes from combining $\llbracket -\gamma,0\rrbracket$ with 
        \begin{align*}
            \{\{j-\sht(j)\mid 1\leq j\leq a\}\}\expl{=}{\autoref{lem:sht}.\autoref{lem:sht_border}} \llbracket 1,a\rrbracket.
        \end{align*}
        Similarly,
        \begin{align*}
            \{\{j-\sht(j)\mid p-a\leq j\leq p\}\}\expl{=}{\autoref{lem:sht}.\autoref{lem:sht_border}} \llbracket -\gamma-a+N_1,-\gamma+N_1-1\rrbracket,
        \end{align*}
        explaining the appearence of the latter in the statement of \autoref{lem:LHS}. In the same vein, \autoref{lem:sht}.\autoref{lem:sht_constant} gives that
        \begin{align*}
            \{\{j-\sht(j)\mid b_m+1\leq j\leq a_{m+1}\}\}=\llbracket b_m+1-\sht(a_{m+1}),a_{m+1}-\sht(a_{m+1})\rrbracket
        \end{align*}
        for all $0\leq m<n$. Piecing everything together, the statement follows.
    \end{proof}

    Now we turn to $\sR$; the idea is to partition it according to the partition of $F$ into the $F^{(j)}$'s. Note that by the above discussion, we have
    \begin{align*}
        F=\bigsqcup_{\substack{0\leq m\leq n\\ a_m\leq j\leq b_m}} F^{(j)}\sqcup \underbrace{F^{(p)}}_{=\{1\}},
    \end{align*}
    and all the $F^{(j)}$'s appearing here are non--empty. Hence for $j\in\llbracket a_m,b_m\rrbracket$ or $j=p$ we define
    \begin{align*}
        \sR_j\coloneqq\left\{\left\{\left.\theta\left(\frac{s}{r}\right)+i\ \right|\ \frac{s}{r}\in F^{(j)},\ \gcd(s,r)=1,\ 0\leq i<N_r\right\}\right\},
    \end{align*}
    which thus yields a multiset--partition of $\sR$. We start with the following description.
    
    \begin{lemma}\label{lem:partition}
       For $0\leq m\leq n$ and $j\in\llbracket a_m,b_m\rrbracket$, we have
        \begin{align*}
            \sR_j=\llbracket \theta(y_{j,k_j}),\theta(y_{j,1})+N_{r_{j,1}}-1\rrbracket.
        \end{align*}
        Also, we have
        \begin{align*}
            \sR_p=\llbracket-\gamma,-\gamma+N_1-1\rrbracket.
        \end{align*}
    \end{lemma}

    \begin{proof}
        The description of $\sR_p$ is straightforward from $F^{(p)}=\{1\}$ and $\theta(1)=-\gamma$. So take $j\in\llbracket a_m,b_m\rrbracket$. Notice that the restriction $\theta_j\coloneqq\theta|_{[j/p,(j+1)/p)}$ is given by the formula $\theta_j(y)=1+j-t(y)$, where
        \begin{align*}
            t(y)\coloneqq\sum_{(\lambda,i)\in\bfI^*}\lfloor iy\rfloor.
        \end{align*}
        Note that $t$ is right--continuous and the discontinuities of $t$ on $[j/p,(j+1)/p)$ occur precisely at the elements of $F^{(j)}$. For $1<k\leq k_j$, we then observe that
        \begin{align*}
            t(y_{j,k-1})=\lim_{y\underset{<}{\to} y_{j,k}}t(y)=\sum_{(\lambda,i)\in\bfI^*}\lim_{y\underset{<}{\to} y_{j,k}}\lfloor iy\rfloor=\sum_{(\lambda,i)\in\bfI^*}\lfloor i\cdot y_{j,k}\rfloor-\bone_{r_{j,k}|i}=t(y_{j,k})-N_{r_{j,k}},
        \end{align*}
        where $\bone_{r_{j,k}|i}$ is $1$ if $r_{j,k}$ divides $i$ and $0$ otherwise. Hence we find that
        \begin{align*}
            \theta(y_{j,k})=\theta(y_{j,k-1})-N_{r_{j,k}},
        \end{align*}
        and in particular $\theta(y_{j,k})<\theta(y_{j,k-1})$ for all $k$. Thus it follows that
        \begin{align*}
            \sR_j&=\biguplus_{1\leq k \leq k_j}\llbracket \theta(y_{j,k}),\theta(y_{j,k})+N_{r_{j,k}}-1\rrbracket \\
            &=\llbracket \theta(y_{j,1}),\theta(y_{j,1})+N_{r_{j,1}}-1\rrbracket \uplus\biguplus_{2\leq k \leq k_j}\llbracket \theta(y_{j,k}),\theta(y_{j,k-1})-1\rrbracket \\
            &=\llbracket \theta(y_{j,1}),\theta(y_{j,1})+N_{r_{j,1}}-1\rrbracket \uplus \llbracket \theta(y_{j,k_j}),\theta(y_{j,1})-1\rrbracket \\
            &=\llbracket \theta(y_{j,k_j}),\theta(y_{j,1})+N_{r_{j,1}}-1\rrbracket 
        \end{align*}
        which is what we wanted to prove.
    \end{proof}

    Now we relate the description in \autoref{lem:partition} to $\sht$.

    \begin{lemma}\label{lem:middle_interval}
        For all $j\in\llbracket a_m,b_m\rrbracket$, we have
        \begin{align*}
            \sR_j=\llbracket j+1-\sht(j+1),j-\sht(j)\rrbracket. 
        \end{align*}
    \end{lemma}

    \begin{proof}
        Note that $\theta$ is right--continuous and the discontinuities of $\theta$ occur precisely at elements of $F\cup (1/p)\bZ$. As $y_{j,k_j}$ is the largest element of $F\cup(1/p)\bZ$ which is strictly smaller than $(j+1)/p$, we then see that
        \begin{align*}
            \theta(y_{j,k_j})&=\lim_{y\underset{<}{\to} \frac{j+1}{p}}\theta(y)=1-\lim_{y\underset{<}{\to} \frac{j+1}{p}}\lfloor y\rfloor+\lim_{y\underset{<}{\to} \frac{j+1}{p}}\lfloor py\rfloor-\sum_{(\lambda,i)\in\bfI^{*}}\lim_{y\underset{<}{\to} \frac{j+1}{p}}\lfloor iy\rfloor\\
            &=1+j-\sum_{(\lambda,i)\in\bfI^*}\left\lfloor i\frac{j+1}{p}\right\rfloor=1+j-\sht(j+1).
        \end{align*}
        On the other hand, as $j/p$ is the largest element of $F\cup(1/p)\bZ$ which is strictly smaller than $y_{j,1}$, we have
        \begin{align*}
            \theta(j/p)=\lim_{y\underset{<}{\to} y_{1,j}}\theta(y)=\theta(y_{1,j})+N_{r_{1,j}}.
        \end{align*}
        Finally, it follows from the definition of $\theta$ that $\theta(j/p)=1+j-\sht(j)$.
        Hence we conclude that
        $\theta(y_{1,j})+N_{r_{1,j}}-1=j-\sht(j)$.
    \end{proof}

    Piecing everything together, we are now ready to prove \autoref{conj:equality_of_multisets}.

    \begin{proof}[\protect{Proof of \autoref{conj:equality_of_multisets}}]
        We have
        \begin{align*}
            \sR=\sR_p\uplus\biguplus_{\substack{0\leq m\leq n\\ a_m\leq j\leq b_m}} \sR_j.
        \end{align*}
        Now as $\{j-\sht(j)\}_{a_m\leq j\leq b_m+1}$ is non--increasing by \autoref{lem:sht}.\autoref{lem:sht_non_increasing}, it follows from \autoref{lem:middle_interval} that
        \begin{align*}
            \biguplus_{a_m\leq j\leq b_m} \sR_j&=\llbracket b_m+1-\sht(b_m+1),a_m-\sht(a_m)\rrbracket\uplus\{\{j-\sht(j)\mid a_m<j\leq b_m\}\}.
        \end{align*}
        Hence it follows that
        \begin{align*}
            \sR=\{\{j-\sht(j)\mid 0\leq m\leq n,\  a_m<j\leq b_m\}\}\uplus\llbracket -\gamma,-\gamma+N_1-1\rrbracket\\
            \uplus\biguplus_{0\leq m\leq n}\llbracket b_m+1-\sht(b_m+1),a_m-\sht(a_m)\rrbracket.
        \end{align*}
        To conclude that $\sL=\sR$, in view of \autoref{lem:LHS}, we are left to show that
        \begin{equation}\label{eq:remaining}
        \begin{split}
            \left(\biguplus_{0\leq m<n}\llbracket b_m+1-\sht(a_{m+1}),a_{m+1}-\sht(a_{m+1})\rrbracket\right)\uplus \llbracket -\gamma-a+N_1,-\gamma+N_1-1\rrbracket\uplus\llbracket -\gamma,a\rrbracket\\
            =\left(\biguplus_{0\leq m\leq n}\llbracket b_m+1-\sht(b_m+1),a_m-\sht(a_m)\rrbracket\right)\uplus \llbracket -\gamma,-\gamma+N_1-1\rrbracket,
        \end{split}
        \end{equation}
        as we can cancel $\{\{j-\sht(j)\mid 0\leq m\leq n,\  a_m<j\leq b_m\}\}$ from both sides. To see more clearly what is going on, we introduce the following notation for $0\leq m\leq n$:
        \begin{align*}
            A_m\coloneqq a_m-\sht(a_m),\quad B_m\coloneqq b_m+1&-\sht(b_m+1)\expl{=}{\autoref{lem:sht}.\autoref{lem:sht_constant}}b_m+1-\sht(a_{m+1}),\\
            A_{n+1}\coloneqq -\gamma+N_1-1,\quad B_{n+1}&\coloneqq -\gamma,\quad A_{n+2}\coloneqq A_0=a.
        \end{align*}
        Note that by \autoref{lem:sht}, every interval appearing in \autoref{eq:remaining} is non--empty, i.e., $B_m\leq A_m$ and $B_m\leq A_{m+1}$ for all $0\leq m\leq n+1$. With this notation, \autoref{eq:remaining} translates to
        \begin{align*}
            \biguplus_{0\leq m\leq n+1}\llbracket B_m,A_{m+1}\rrbracket =\biguplus_{0\leq m\leq n+1}\llbracket B_m,A_m\rrbracket,
        \end{align*}
        where we used that $B_n= -\gamma-a+N_1$ by \autoref{lem:sht}. This equality of multisets is true whenever all intervals appearing are non--empty, as is shown in \autoref{lem:union_intervals}.
    \end{proof}

    \begin{lemma}\label{lem:union_intervals}
        Let $A_0,\ldots,A_N$ and $B_0,\ldots,B_N$ be integers and write $A_{N+1}=A_0$. Suppose that $B_m\leq\min\{A_m,A_{m+1}\}$ for all $0\leq m\leq N$. Then
        \begin{align*}
            \biguplus_{0\leq m\leq N}\llbracket B_m,A_m\rrbracket=\biguplus_{0\leq m\leq N}\llbracket B_m,A_{m+1}\rrbracket
        \end{align*}
    \end{lemma}

    \begin{proof}
        We have to show that
        \begin{align*}
            f\coloneqq \sum_{0\leq m\leq N}\bone_{\llbracket B_m,A_m\rrbracket}-\bone_{\llbracket B_m,A_{m+1}\rrbracket}\equiv 0.
        \end{align*}
        Note that
        \begin{align*}
            \bone_{\llbracket B_m,A_m\rrbracket}-\bone_{\llbracket B_m,A_{m+1}\rrbracket}=\bone_{\rrbracket -\infty,A_m\rrbracket}-\bone_{\rrbracket -\infty,A_{m+1}\rrbracket},
        \end{align*}
        as both are equal to $\sgn(A_m-A_{m+1})$ on $\rrbracket\min\{A_m,A_{m+1}\},\max\{A_m,A_{m+1}\}\rrbracket$ and equal to $0$ everywhere else (here we use that $B_m\leq \min\{A_m,A_{m+1}\}$). Hence it follows that
        \begin{align*}
            f&=\sum_{0\leq m\leq N}\bone_{\rrbracket -\infty,A_m\rrbracket}-\bone_{\rrbracket -\infty,A_{m+1}\rrbracket}\\
            &=\bone_{\rrbracket -\infty,A_0\rrbracket}-\bone_{\rrbracket -\infty,A_{N+1}\rrbracket}\\
            &=0,
        \end{align*}
        as $A_{N+1}=A_0$.
    \end{proof}

\bibliographystyle{alpha}
\bibliography{Bibliography}

\newcommand{\etalchar}[1]{$^{#1}$}
\begin{thebibliography}{KKMSD73}

\bibitem[Bat99]{Batyrev_NA_integrals_and_stringy_Euler_numbers}
Victor~V. Batyrev.
\newblock Non-archimedean integrals and stringy {Euler} numbers of log-terminal
  pairs.
\newblock {\em J. Eur. Math. Soc. (JEMS)}, 1(1):5--33, 1999.

\bibitem[Ber17]{Bergh_Functorial_destackification}
Daniel Bergh.
\newblock Functorial destackification of tame stacks with abelian stabilisers.
\newblock {\em Compos. Math.}, 153(6):1257--1315, 2017.

\bibitem[CLNS18]{Chambert-Loir_Nicaise_Sebag_Motivic_integration}
Antoine Chambert-Loir, Johannes Nicaise, and Julien Sebag.
\newblock {\em Motivic integration}, volume 325 of {\em Prog. Math.}
\newblock New York, NY: Birkh{\"a}user, 2018.

\bibitem[CLS11]{Cox-Little-Schenck}
David~A. Cox, John~B. Little, and Henry~K. Schenck.
\newblock {\em Toric varieties}, volume 124 of {\em Grad. Stud. Math.}
\newblock Providence, RI: American Mathematical Society (AMS), 2011.

\bibitem[DGA{\etalchar{+}}11]{SGA3_I}
Michel Demazure, Alexander Grothendieck, M.~Artin, J.-E. Bertin, P.~Gabriel,
  M.~Raynaud, and J.-P. Serre, editors.
\newblock {\em S{\'e}minaire de g{\'e}om{\'e}trie alg{\'e}brique du {Bois}
  {Marie} 1962-64. {Sch{\'e}mas} en groupes ({SGA} 3). {Tome} {I}:
  {Propri{\'e}t{\'e}s} g{\'e}n{\'e}rales des sch{\'e}mas en groupes.}, volume~7
  of {\em Doc. Math. (SMF)}.
\newblock Paris: Soci{\'e}t{\'e} Math{\'e}matique de France, new annotated
  edition of the 1970 original published bei {Springer} edition, 2011.

\bibitem[GNT19]{Gongyo-Nakamura-Tanaka}
Yoshinori Gongyo, Yusuke Nakamura, and Hiromu Tanaka.
\newblock Rational points on log {Fano} threefolds over a finite field.
\newblock {\em J. Eur. Math. Soc. (JEMS)}, 21(12):3759--3795, 2019.

\bibitem[Gro66]{EGA_IV.3}
A.~Grothendieck.
\newblock \'{E}l\'{e}ments de g\'{e}om\'{e}trie alg\'{e}brique. {IV}. \'{E}tude
  locale des sch\'{e}mas et des morphismes de sch\'{e}mas. {III}.
\newblock {\em Inst. Hautes \'{E}tudes Sci. Publ. Math.}, (28):255, 1966.

\bibitem[IR96]{Ito_Reid_McKay_correspondence_SL3}
Yukari Ito and Miles Reid.
\newblock The {McKay} correspondence for finite subgroups of
  {{\(SL(3,\mathbb{C})\)}}.
\newblock In {\em Higher dimensional complex varieties. Proceedings of the
  international conference, Trento, Italy, June 15--24, 1994}, pages 221--240.
  Berlin: Walter de Gruyter, 1996.

\bibitem[KKMSD73]{KKMSD_Toroidal_embeddings}
G.~Kempf, F.~Knudsen, D.~Mumford, and Bernard Saint-Donat.
\newblock {\em Toroidal embeddings. {I}}, volume 339 of {\em Lect. Notes Math.}
\newblock Springer, Cham, 1973.

\bibitem[Kol13]{Kollar_Singularities_of_the_minimal_model_program}
J\'{a}nos Koll\'{a}r.
\newblock {\em Singularities of the minimal model program}, volume 200 of {\em
  Cambridge Tracts in Mathematics}.
\newblock Cambridge University Press, Cambridge, 2013.
\newblock With a collaboration of S\'{a}ndor Kov\'{a}cs.

\bibitem[KT23]{Kresch_Tschinkel_Birat_geom_DM_stacks}
Andrew Kresch and Yuri Tschinkel.
\newblock Birational geometry of {Deligne}-{Mumford} stacks.
\newblock Preprint, {arXiv}:2312.14061 [math.{AG}] (2023), 2023.

\bibitem[LMM21]{Liedtke_Martin_Matsumoto_Lrq_sing}
Christian Liedtke, Gebhard Martin, and Yuya Matsumoto.
\newblock Linearly {R}eductive {Q}uotient {S}ingularities.
\newblock {\em ArXiv e-print, arXiv:2102.01067v2. To be published in
  Ast\'{e}risque.}, 2021.

\bibitem[MI21]{Miyanishi_Ito_Algebraic_surfaces_in_pos_char}
Masayoshi Miyanishi and Hiroyuki Ito.
\newblock {\em Algebraic surfaces in positive characteristics. {Purely}
  inseparable phenomena in curves and surfaces}.
\newblock Hackensack, NJ: World Scientific, 2021.

\bibitem[Ols06]{Olsson_Hom_stacks}
Martin~C. Olsson.
\newblock {{\(\underline{\mathrm{Hom}}\)}}-stacks and restriction of scalars.
\newblock {\em Duke Math. J.}, 134(1):139--164, 2006.

\bibitem[Pos25a]{Posva_Pathological_MMP_sing}
Quentin Posva.
\newblock Pathological {MMP} singularities as {{\(\alpha_p\)}}-quotients.
\newblock {\em Forum Math. Sigma}, 13:29, 2025.
\newblock Id/No e185.

\bibitem[Pos25b]{Posva_Resolution_1-foliations}
Quentin Posva.
\newblock Resolution of $1$-foliations singularities on surfaces and
  threefolds.
\newblock Preprint, {arXiv}:2405.05735 [math.{AG}] (2025). To appear in Algebra
  and Number theory, 2025.

\bibitem[Pos26]{Posva_Singularities_quotients_by_1-foliations}
Quentin Posva.
\newblock On the singularities of quotients by 1-foliations.
\newblock {\em Nagoya Math. J.}, 261:41, 2026.
\newblock Id/No e6.

\bibitem[QR22]{Quek_Rydh_Weighted_blowups}
Ming~Hao Quek and David Rydh.
\newblock Weighted blow-ups.
\newblock {\em available on the second author's webpage}, 2022.

\bibitem[Rei19]{Reid_Tate-Oort_group}
Miles Reid.
\newblock The {Tate}-{Oort} group scheme {{\(\mathbb{TO}_p\)}}.
\newblock {\em Proc. Steklov Inst. Math.}, 307:245--266, 2019.

\bibitem[Tot19]{Totaro_Failure_Kodaira_vanishing}
Burt Totaro.
\newblock The failure of {Kodaira} vanishing for {Fano} varieties, and terminal
  singularities that are not {Cohen}-{Macaulay}.
\newblock {\em J. Algebr. Geom.}, 28(4):751--771, 2019.

\bibitem[Tot24]{Totaro_Terminal_3folds_not_CM}
Burt Totaro.
\newblock Terminal 3-folds that are not {Cohen}-{Macaulay}.
\newblock Preprint, {arXiv}:2407.02608 [math.{AG}] (2024), 2024.

\bibitem[TY20]{Tonini_Yasuda_Motivic_McKay_cor_for_alpha_p}
Fabio Tonini and Takehiko Yasuda.
\newblock Notes on the motivic {M}c{K}ay correspondence for the group scheme
  {$\alpha_p$}.
\newblock In {\em Singularities---{K}agoshima 2017}, pages 215--233. World Sci.
  Publ., Hackensack, NJ, [2020] \copyright 2020.

\bibitem[W{\l}o22]{Wlodarczyk_functorial_resolution}
Jaros{\l}aw W{\l}odarczyk.
\newblock Functorial resolution except for toroidal locus. {Toroidal}
  compactification.
\newblock {\em Adv. Math.}, 407:103, 2022.
\newblock Id/No 108551.

\bibitem[{Wol}]{Mathematica}
{Wolfram Research, Inc.}
\newblock Mathematica, {V}ersion 14.2.
\newblock Champaign, IL, 2024.

\bibitem[Yam24]{YamamotoThesis2024}
Yudai Yamamoto.
\newblock The space of twisted arcs for the group scheme \(\alpha_p\)
  ({J}apanese).
\newblock Master's thesis, Osaka University, 2024.

\bibitem[Yas06]{Yasuda_Motivic_integration_over_DM_stacks}
Takehiko Yasuda.
\newblock Motivic integration over {Deligne}-{Mumford} stacks.
\newblock {\em Adv. Math.}, 207(2):707--761, 2006.

\bibitem[Yas14]{Yasuda_p-cyclic_McKay_correspondence}
Takehiko Yasuda.
\newblock The {{\(p\)}}-cyclic {McKay} correspondence via motivic integration.
\newblock {\em Compos. Math.}, 150(7):1125--1168, 2014.

\bibitem[Yas19]{Yasuda_Discrepancies_of_p_cyclic_quotient}
Takehiko Yasuda.
\newblock Discrepancies of {{\(p\)}}-cyclic quotient varieties.
\newblock {\em J. Math. Sci., Tokyo}, 26(1):1--14, 2019.

\end{thebibliography}

\end{document}